\documentclass[poms,final,nonblindrev]{arxiv} 

\usepackage{graphicx}
\usepackage{subfigure, epsfig}
\usepackage{natbib}
\usepackage{newtxtext}
\usepackage{newtxmath}
\usepackage{float}
\usepackage{multirow}

\usepackage{calrsfs}
\usepackage{amsmath, amssymb, bm}
\usepackage{longtable}

 \bibpunct[, ]{(}{)}{,}{a}{}{,}%
 \def\bibfont{\small}%

\TheoremsNumberedThrough    
\ECRepeatTheorems

\EquationsNumberedThrough    
\begin{document}

\RUNTITLE{Green Capacity Planning by AI+DRO}

\TITLE{Green Manufacturing Capacity Planning by Integrating Distributionally Robust Optimization and Generative AI}

\ARTICLEAUTHORS{%
\AUTHOR{Xin Zhou}
\AFF{Antai College of Economics and Management,Shanghai Jiao Tong University, \EMAIL{shaynezhou@sjtu.edu.cn}}
\AUTHOR{Zhengsong Lu}
\AFF{Department of Industrial Engineering, University of Pittsburgh, \EMAIL{zs.lu@pitt.edu}}
\AUTHOR{Bo Zeng}
\AFF{Department of Industrial Engineering, University of Pittsburgh, \EMAIL{bzeng@pitt.edu}}
\AUTHOR{Na Geng*}
\AFF{Sino-US Global Logistics Institute, Shanghai Jiao Tong University, \EMAIL{gengna@sjtu.edu.cn}}
}

\ABSTRACT{
Green manufacturing has become a strategic priority for many firms seeking to address sustainability and social responsibility, while improving production efficiency and profitability. 
However, integrating green technologies and renewable energy unavoidably introduces climate-related randomness that affects both product demand and renewable energy generation, underscoring the need for coordinated planning of production capacity and renewable energy development. 
To address this challenge, we develop a comprehensive two-stage distributionally robust optimization (DRO) model for green manufacturing capacity planning in a multi-factory, multi-capacity, and multi-product setting, based on an ambiguity set constructed by a data-driven clustering technique that leverages historical data of different availabilities and qualities. 
To handle the computational challenges of practical instances, an effective generative AI network is integrated into an exact decomposition algorithm, through a novel encoding/decoding scheme designed to provide the AI model with structurally informative training data and to convert AI-generated outputs into algorithm-accessible formats. 
Experimental results on real-world instances demonstrate that the proposed DRO approach achieves strong economic performance and robust feasibility under demand and renewable generation uncertainty, while also significantly improving computational efficiency and solution consistency relative to the standard approaches. 
Furthermore, our results highlight the managerial value of integrating green technology adoption with coordinated capacity planning to better utilize renewable energy and align production efficiency with sustainability and corporate social responsibility objectives.}

\KEYWORDS{green manufacturing, capacity planning, distributionally robust optimization, generative AI}

\maketitle
\section{Introduction}\label{sect_intro}
The manufacturing sector is one of the primary contributors to global energy consumption and carbon emissions, thereby underscoring its critical role in advancing global climate goals. 
According to the World Energy Outlook 2024 \citep{laura_world_2024}, the manufacturing industry accounted for more than one-third of global energy consumption in 2023. Moreover, within the lifecycle emissions of final products, the manufacturing stage alone was responsible for over half of the associated CO$_2$ emissions, further emphasizing the sector’s pivotal role in achieving sustainability. 
Over the past 20 years, the confluence of stringent carbon-emission regulations and consumers’ growing demand for eco-friendly products has introduced unprecedented challenges for manufacturers, including intensified market competition and elevated operational costs. 
To navigate this tension, firms worldwide have accelerated their transition toward green manufacturing, with green practices playing a central role in this transformation.
For instance, as early as 2022, the BMW Group announced that it had fully transitioned to 100\% renewable energy for its production processes in Europe. 
Similarly, the under-construction Tesla Giga Shanghai factory is projected to deploy over 6 MW of rooftop photovoltaic (PV) capacity, highlighting the growing scale of renewable energy integration in manufacturing.

Against this backdrop, jointly considering production capacity and renewable energy planning plays a critical role in enhancing the sustainability, efficiency, and resilience of green manufacturing systems. 
Production capacity planning determines manufacturing scale, resource allocation, and long-term operational efficiency. 
However, without proper alignment with renewable energy availability, manufacturers may either over-invest in renewable infrastructure or under-provision capacity, thereby undermining sustainability objectives. 
At the same time, renewable energy output is inherently subject to climate variability and cannot be treated as a resource with fixed capacity. 
For example, peak sunshine hours, a key climatic factor, serves as a dominant determinant of solar energy production \citep{rehman_cost_2007}. 
This climate-driven uncertainty in renewable energy generation, when coupled with product demand uncertainty, significantly complicates strategic and operational decision-making. 
Consequently, joint planning of production capacity and renewable energy development can yield substantial strategic benefits by enabling adaptive production scheduling and coordinated capacity expansion alongside renewable energy deployment.

Capacity planning under uncertainty has been a central focus of production and operations management research for decades.
This research field has primarily relied on two traditional approaches: stochastic programming (SP) and robust optimization (RO), with distinct strengths and limitations.   
As a well-established paradigm, SP assumes that the underlying uncertainty can be accurately characterized by a known probability distribution. However, such an assumption might not hold due to limited data availability or poor data quality.  
Deviations from the assumed distribution can lead to solutions with degraded performance, causing SP to yield overly optimistic decisions \citep{shapiro_chapter_2021}.
RO, as a more recent alternative, focuses on worst-case performance without requiring distributional information for uncertain parameters \citep{ben-tal_robust_2009}. This feature enhances its applicability in settings where historical data are scarce for obtaining dependable distributional information or where system reliability is absolutely critical. However, RO is often criticized for yielding overly conservative decisions that lead to reduced economic efficiency.

In recent years, distributionally robust optimization (DRO) has emerged as a flexible and powerful framework for modeling and mitigating uncertainty \citep{delage_distributionally_2010,wiesemann_distributionally_2014}.
Rather than relying on a single presumed probability distribution, DRO constructs an ambiguity set that contains a family of possible distributions. It then seeks solutions that explicitly account for this distributional uncertainty, enabling a systematic trade-off between robustness and economic performance. From a unifying perspective, SP and RO can be interpreted as special cases of DRO. 
Specifically, DRO reduces to SP when the ambiguity set consists of a single known distribution and reduces to RO when the ambiguity set contains all Dirac delta distributions supported on its sample space. Consequently, DRO provides substantial flexibility and modeling capability for handling different types and levels of uncertainty.

Nevertheless, DRO, especially the two-stage DRO, still faces non-trivial computational challenges when applied to large-scale and sophisticated instances. To overcome these challenges, several solution strategies have been developed in the literature, including reformulations into mixed-integer (nonlinear) programming (MIP) equivalents and decomposition-based algorithms. Nevertheless, these approaches are often restricted to special cases or exhibit limited computational scalability and consistency in practical settings. Compared to its SP and RO counterparts, two-stage DRO has seen far fewer real-world applications. 
Meanwhile, the past decade has witnessed growing interest in leveraging advanced artificial intelligence (AI) techniques to enhance optimization algorithms, especially decomposition-based ones. Non-trivial progress has been made in developing AI-assisted Benders decomposition and column generation (CG). 
Yet, to the best of our knowledge, little research has explored the integration of AI techniques and DRO. 
This gap suggests that the development of AI-enhanced solution methods for DRO is both timely and essential to unlock the applicability of DRO across a wide range of real-world problems.

In green manufacturing systems, DRO is well suited due to the nature and availability of data: while historical product demand data are often sparse, fragmented, or subject to rapid changes, long-term climate data (e.g., seasonal solar irradiance and temperature patterns) are typically more comprehensive, consistent, and accessible. 
This setting allows us, under the DRO framework, to employ data-driven clustering techniques to generate representative regimes (with probabilities) to capture reliable climatic variability and adopt ambiguity sets to model and hedge against less-understood randomness in product demand. Then, we develop a two-stage DRO to solve the green manufacturing capacity planning problem. 
To address the associated computational challenges, we develop a novel AI-enhanced exact algorithm that achieves scalability and consistency. It differs fundamentally from prior AI-assisted optimization approaches by embedding structural properties directly into the learning process via an encoding/decoding scheme, rather than treating AI as an external black-box component.
In turn, the well-trained AI model is instrumental in accelerating convergence and stabilizing algorithmic behavior.

The remainder of this paper is organized as follows. 
Section 2 reviews the related literature on green manufacturing, DRO, and the application of machine learning/AI tools in optimization algorithms. 
Section 3 introduces the problem context and presents a two-stage DRO model for green manufacturing capacity planning. Section 4 develops the solution methodology and proposes a generative AI–enhanced solution architecture. Section 5 details the encoding/decoding scheme and the training procedures of the AI network. Section 6 presents the experimental results and analysis regarding both the DRO model and the algorithmic performance. Finally, Section 7 concludes the paper.

\section{Literature Review}

\subsection{Strategic Planning for Green Manufacturing}

As enterprises increasingly prioritize the “triple bottom line” (economic, environmental, and social performance), green manufacturing has evolved from a niche practice to a core strategic imperative for manufacturers, with studies demonstrating its ability to create mutual value for firms and society \citep{lee_socially_2018,de_haas_managerial_2025}. 
To date, green manufacturing is a central focus of strategic planning research in the manufacturing sector.

In the context of strategic planning, carbon emission regulations (e.g., carbon taxes, cap-and-trade) have emerged as key drivers of strategic capacity decisions, as they directly impact investment costs and operational flexibility. \citet{drake_technology_2016} linked regulations to dual strategic choices: not only capacity configuration but also green technology adoption, finding that stringent carbon policies accelerate the replacement of conventional equipment with low-carbon alternatives. 
\citet{song_capacity_2017} showed that both carbon emission price and carbon tax significantly affect capacity expansion and production decisions, while \citet{huang_carbon_2021} extended this by demonstrating how regulations prompt capacity shifting (e.g., relocating production to regions with lower carbon intensity).  
This stream primarily focuses on capacity adjustments for emission reduction rather than integrating green technologies into broader production capacity planning.

Another main research stream of green manufacturing is renewable energy integration, which focuses on designing manufacturing systems that leverage distributed renewable sources (e.g., solar, wind) to reduce carbon footprint. \citet{pham_multi-site_2019} optimized the location and size of renewable energy facilities for multi-site manufacturing networks. A two-stage SP model was proposed to attain net-zero energy production-logistics operations through renewable microgrid generation. 
\citet{bertsimas_decarbonizing_2023} developed a framework to align renewable energy procurement with production schedules. A notable exception is \citet{zhou_towards_2025}, which considered both renewable energy capacity and production capacity. Yet, their model is restricted to a simplified setting with a single product and a single site and, therefore, does not capture the multi-factory, multi-category realities of modern manufacturing. 
Moreover, \citet{lin_designing_2020} showed that green technology adoption (e.g., energy-efficient machinery) is critical for green manufacturing, as it helps scale the benefits of renewable energy. 
Nevertheless, existing studies have not integrated green technology adoption, renewable energy capacity planning, and production capacity planning within a general strategic planning framework.

\subsection{Uncertainty Modeling in Manufacturing Strategic Planning}
Addressing operations-centric uncertainty is a fundamental concern in strategic planning research. 
Traditionally, demand uncertainty has been the primary source of randomness in strategic planning decisions of manufacturing systems \citep{martinez-costa_review_2014,lyu_capacity_2019}. 
As mentioned in Section \ref{sect_intro}, green manufacturing introduces an additional layer of complexity, namely climate-driven renewable energy generation uncertainty, where the product demand may be correlated with climatic factors. 
In what follows, we review how three popular methodologies, i.e., SP, RO, and DRO, have been used to address uncertainty and discuss their limitations for decision-making in green manufacturing systems.

SP has long served as a primary approach for capacity planning under uncertainty, as it directly models random parameters via a probability distribution estimated from historical data. 
For example, \citet{lee_proactive_2014} optimized short-term capacity planning using an SP model by incorporating diminishing marginal returns on investment. For green manufacturing, \citet{song_capacity_2017} extended SP to include carbon tax uncertainty, showing that flexible capacity (e.g., modular production lines) reduces regulatory risk. Most recently, \citet{yu_value_2024} developed a multi-stage SP framework for dynamic capacity expansion, capturing the demand evolution over time.
However, SP has critical limitations, as it requires large amounts of high-quality historical data to (1) estimate and validate distributional parameters and (2) to capture correlations among random factors that are otherwise assumed to be independent in general.
 
RO avoids reliance on probabilistic distributions by directly optimizing worst-case performance over a predefined uncertainty set, thus requiring no distributional assumptions. 
In the green manufacturing literature, \citet{zhou_towards_2025} employed RO to jointly plan renewable energy and production capacities under demand and carbon price uncertainty, demonstrating that RO can reduce the risk of failing to meet sustainability targets. 
Similarly, \citet{gao_adjustable_2025} applied RO to reverse supply chains, a key component of green manufacturing systems. 
The primary critique of RO lies in its inherent conservatism, which often leads to excessive over-investment \citep{jakubovskis_flexible_2017}. This issue is particularly problematic for manufacturers seeking to balance sustainability objectives with economic profitability.

As a generalization of SP and RO, DRO has been applied to problems with demand uncertainty, including lot-sizing \citep{zhang_distributionally_2016} and facility location \citep{saif_data-driven_2021, liu_robust_2022} problems. \citet{basciftci_distributionally_2021} further extended DRO to random, endogenous demand. 
As previously noted, once distributed renewable energy is incorporated, manufacturing systems inevitably face climate-driven renewable energy generation uncertainties \citep{golari_multistage_2017, pham_multi-site_2019}, which necessitates handling more complex ambiguity sets involving multiple sources of uncertainty. 
Owing to both modeling and computational challenges, as well as the inherent operational complexity of green manufacturing systems, existing DRO studies in this domain remain limited.
Specifically:
(1) most studies focus solely on demand uncertainty, ignoring climate-driven renewable generation,
(2) no prior work explicitly models the correlation between climatic factors and demand, and 
(3) integrating multiple sources of uncertainty within multi-stage DRO formulations for large-scale manufacturing networks remains essentially unexplored.

\subsection{AI-Enhanced Optimization Algorithms}

Recent advances in AI have opened new avenues to address the computational complexity of large-scale optimization problems, which are prevalent in real-world manufacturing systems involving thousands of variables and constraints.
Most existing studies enhance exact optimization algorithms, such as branch-and-bound, Benders decomposition, and CG, by incorporating AI-based strategy selection, value approximation, or warm-start mechanisms.

In branch-and-bound algorithms, for example, \citet{alvarez_machine_2017} and \citet{rajabalizadeh_solving_2024} employ ExtraTrees and logistic regression techniques to accelerate node selection.
Similarly, several studies have employed unsupervised learning methods to identify high-quality cuts, thereby enhancing the performance of Benders decomposition algorithms \citep{jia_benders_2021,wu_machine_2025}. 
In studies applying AI to CG, the primary focus is on accelerating the algorithm by improving column selection during each iteration \citep{morabit_machine-learningbased_2021,morabit_machine-learningbased_2023}.
With the introduction of graph neural networks (GNNs) to model linear programs (LP) or MIP \citep{nair_solving_2021}, it has become possible to capture the structure of the master problem and rapidly evaluate the quality of candidate columns.
The effectiveness of GNN-based column selection has been demonstrated by \citet{morabit_machine-learningbased_2021,chi_deep_2022} and \citet{cui_configuration_2025}
These studies include the vehicle routing problem with time windows, the cutting stock problem, and reconfigurable manufacturing systems.
However, the performance of such methods depends on the effectiveness of solving the pricing problem, which entails generating a sufficient number of high-quality columns.
When solving the pricing problem proves costly, an alternative and more straightforward idea is to initialize columns in the master problem using neural networks \citep{kraul_machine_2023,hijazi_all_2024} or heuristic methods \citep{liang_column_2018}.
Yet, previous initialization research lacks the ability to efficiently generate large batches of columns. Moreover, there is no guarantee that the predicted columns are feasible, which necessitates additional validation and post-processing operations. 

\subsection{Research Gaps and Discussion}
To the best of knowledge, there is no comprehensive framework that simultaneously integrates production capacity planning, green technology adoption, and distributed renewable energy considerations for multi-product, multi-site manufacturers. 
Such integration is critically important, as it is essential for these manufacturers to comply with regulations and achieve their sustainability goals. 
In the absence of such a framework, manufacturers will struggle to optimize their operations across different products and sites while adhering to environmental regulations, meeting sustainability targets, and achieving social responsibility goals. 
Moreover, with respect to uncertainty modeling in green manufacturing systems, existing optimization models fail to account for climate-relevant demand and renewable energy generation, as well as the correlations between these sources of uncertainty. 
This oversight is significant, as these uncertainties and their interdependencies can substantially affect the effectiveness of strategic planning decisions in green manufacturing systems. 

With respect to the application of AI in manufacturing, numerous successes have been reported in demand forecasting, process monitoring, and operational diagnostics, given the strong ability of AI models to extract patterns from data and approximate complex functional relationships.
AI models are, however, often limited by their black-box nature and their limited awareness of feasibility boundaries in decision spaces. 
These limitations hinder their direct integration into computational algorithms for decision-making problems, including those encountered in manufacturing systems. 
In particular, systematically exploiting the structural characteristics of decision models to enhance the feasibility and accuracy of AI model's output remains largely unexplored.

To address the aforementioned gaps, we collaborate with a manufacturer operating multiple factories, production lines, and product categories to address a real-world capacity planning problem that involves green technology adoption. 
To reflect the underlying decision-making structure and capture key sources of uncertainty, the planning problem is formulated as a two-stage DRO model, where a data-driven clustering method is employed to capture two critical uncertainties, product demand fluctuations and climate-driven renewable energy generation variability. 
To solve the resulting large-scale problem, we customize a recently proposed exact C\&CG-DRO algorithm \citep{lu_two-stage_2024}. 
Furthermore, through innovative designs and implementations, our model's structural properties are represented through an encoding/decoding scheme to leverage a powerful generative AI model, which when integrated into the algorithm, provides a significant improvement in computational efficiency and consistency.

We note that this study establishes and demonstrates a feature engineering framework to bridge the gap between AI tools and exact optimization algorithms. 
By first uncovering the structural characteristics of the decision problem and then designing and applying an effective data encoding/decoding scheme, we provide the AI model with structurally informative data for learning and transform AI-generated outputs into formats that are compatible with rigorous mathematical optimization.
Specific to solving a DRO model, we introduce a scenario-to-image encoding method that enables an AI model to learn meaningful, latent patterns among critical scenarios. 
The generated images can then be decoded to construct a high-quality scenario set that can be directly integrated into the C\&CG-DRO algorithm, leading to faster convergence without sacrificing solution exactness. 
Most existing AI models are not directly suitable for optimization-based decision problems, and research on feature engineering, particularly problem-specific encoding/decoding schemes that integrate AI with optimization methodologies, remains very limited. 
Current studies largely rely on standard feature identification and representation techniques, with little customization or tailoring to the key structural properties of an underlying optimization model. 
This gap suggests an important research direction is the development of customizable feature-engineering techniques for integrating AI and optimization algorithms.
Such techniques would leverage the learning capabilities of AI to uncover hidden structural patterns, while ensuring that the generated outputs are compatible with, and valuable to, exact and approximation solution procedures.

\section{A DRO Model for Green Manufacturing Capacity Planning}
In this section, we first introduce the capacity planning problem for green manufacturing and present a two-stage DRO model. 
A data-driven clustering scheme based on historical climate data is proposed to capture the inherent dependence between climate-relevant demand and renewable energy generation. 
Additionally, we establish several structural properties of the DRO model that provide modeling insights and computational advantages, and facilitate the development of our AI-enhanced solution approach.  

\subsection{The Strategic Stage Description and Model}
Consider a two-stage capacity planning problem for a multi-site, multi-capacity, multi-product manufacturer. 
Let $\mathcal{T}$ denote the set of periods in the planning horizon, $\mathcal{I} $ the set of factories, $\mathcal{J} $ the set of capacity types, and $\mathcal{K}$ the set of product categories, where each capacity type is capable of producing one or more product categories. 
In the strategic (first) stage, the manufacturer decides on capacity expansion, reduction, and technological upgrading of conventional production lines over $\mathcal{T}$. 
For existing production lines, those with conventional technologies produce only traditional products. 
However, when upgraded to green production lines, they acquire the capacity to produce green products. 
Importantly, this upgrade is irreversible; once a production line is made green, it cannot be reverted to its previous state. 
To enable the production of green products, renewable energy generation capacity, such as on-site rooftop PV systems, must be installed with investment cost $I^{R}_{i}$. 
If available production capacity is insufficient, additional production lines need to be set up. Similar to existing lines, these new lines also require technological upgrading and on-site renewable energy generation to produce green products. Conversely, if the available capacity exceeds demand, some conventional production lines can be shut down to cut costs. 
Next, we describe the mathematical representation of the first-stage decisions, assuming all production lines are conventional ones before planning. 

Let $X_{ijt} \in \mathbb{N}$ denote the number of production lines of capacity type $j$ in period $t$ at factory $i$. 
To represent the capacity adjustments to the number of production lines in $t$, let $X^+_{ijt} \in \mathbb{N}$ denote the number of production lines (with a unit expansion cost $I^+_{ij}$) added in $t$, and $X^-_{ijt} \in \mathbb{N}$ the number of production lines (with a unit termination cost $I^-_{ij}$) terminated in $t$. The capacity balance constraint across consecutive periods is given by 
 \begin{equation}\label{eq:1}
    X_{ij(t+1)} = X_{ijt}+X^{+}_{ijt}-X^{-}_{ijt},\forall i \in \mathcal{I},\forall j \in \mathcal{J},1 \leq t \leq T-1,
 \end{equation}
with $\hat{x}_{ij}$ being the initial capacity configuration, i.e. $X_{ij1} = \hat{x}_{ij}$.
The number of production lines added and terminated at each factory is constrained by the capacity adjustment limits $f^{+}_{ij}$ and $f^{-}_{ij}$ in each period:
 \begin{equation}\label{eq:2}
    X^{+}_{ijt} \leq f^{+}_{ij},X^{-}_{ijt} \leq f^{-}_{ij},\forall i \in \mathcal{I},\forall j \in \mathcal{J},\forall t \in \mathcal{T}.
 \end{equation}
Let $X^{O}_{ijt} \in \mathbb{N}$ and $X^{N}_{ijt} \in \mathbb{N}$ denote the numbers of production lines with conventional and green technology. The following relationship holds:
 \begin{equation}\label{eq:3}
    X^{O}_{ijt}+X^{N}_{ijt} = X_{ijt},\forall i \in \mathcal{I},\forall j \in \mathcal{J},\forall t \in \mathcal{T}.
 \end{equation}
For the green technology upgrade plan, let $X^{N+}_{ijt} \in \mathbb{N}$ denote the number of production lines upgraded to green technology in period $t$ at factory $i$ for capacity type $j$, incurring a unit cost $I^{N+}_{ij}$. 
The following balance constraint is imposed:
 \begin{equation}\label{eq:4}
    X^{N}_{ij(t+1)} = X^{N}_{ijt}+X^{N+}_{ijt},\forall i \in \mathcal{I},\forall j \in \mathcal{J},1 \leq t \leq T-1,
 \end{equation}
where the initial configurations of green technology production lines are given as $X^{N}_{ij1} = \hat{x}^{N}_{ij}$.

In this study, all $\hat{x}^{N}_{ij}$ are set to 0. 
Moreover, we introduce the variable $X^{BR}_{i} \in \{ 0,1\}$ such that it is equal to $1$ if factory $i$ is equipped with renewable energy generation facilities and $0$ otherwise. 
As discussed, green technology upgrades and renewable energy generation are both necessary for green production. Hence, the following logical constraints between $X^{N+}_{ijt}$ and $X^{BR}_{i}$ are introduced to tighten the formulation. 
 \begin{equation}\label{eq:5}
    X^{N+}_{ijt} \leq M_{0}X^{BR}_{i},\forall i \in \mathcal{I},\forall j \in \mathcal{J},\forall t \in \mathcal{T};
 \end{equation}
 \begin{equation}\label{eq:6}
    X^{BR}_{i} \leq \sum_{\mathcal{J},\mathcal{T}}X^{N+}_{ijt},\forall i \in \mathcal{I},
 \end{equation}
where $M_{0}$ is a positive constant representing an upper bound on the total number of production lines throughout the planning horizon. Let $\boldsymbol{X}$ collectively denote the vector of first-stage decision variables, and $\boldsymbol{C}_{\boldsymbol{X}}$ denote the corresponding coefficient vector in the objective function. 
The feasible set $\mathcal{X}$ of $\boldsymbol{X}$ is defined by constraints \eqref{eq:1}-\eqref{eq:6}.
With the objective of minimizing strategic costs, including the adjustment cost of existing production lines, the upgrade cost of production lines to green technology, and the investment cost of renewable energy facilities, the first-stage model is formulated as follows:
 \begin{equation}\label{eq:7}
    \min_{\boldsymbol{X} \in \mathcal{X}} \boldsymbol{C}_{\boldsymbol{X}}\boldsymbol{X} = \min_{\boldsymbol{X} \in \mathcal{X}} \sum_{\mathcal{I},\mathcal{J},\mathcal{T}}(I^{+}_{ij}X^{+}_{ijt}+I^{-}_{ij}X^{-}_{ijt})+\sum_{\mathcal{I},\mathcal{J},\mathcal{T}}I^{N+}_{ij}X^{N+}_{ijt}+\sum_{\mathcal{I}}I^{R}_{i}X^{BR}_{i}.
 \end{equation}

\subsection{The Operational Stage Description and Model}
In the operational (second) stage, the manufacturer determines production quantities for each product type. 
As products produced on green technology production lines using renewable energy are classified as green products, we let $Y^{NG}_{ijkt} \in \mathbb{R}_{+}$ denote the quantity of green products of type $k$ produced in period $t$ by capacity type $j$ at factory $i$. 
When the renewable energy supply is inadequate, or when production takes place on conventional technology lines, some products must be produced by conventional energy and are classified as traditional products. 
The quantities of such products produced on green technology and conventional technology production lines are denoted as $Y^{NT}_{ijkt} \in \mathbb{R}_{+}$ and $Y^{OT}_{ijkt} \in \mathbb{R}_{+}$, respectively.
The unit production costs on conventional and green technology production lines are denoted as $c^{O}_{ijk}$ and $c^{N}_{ijk}$, respectively.

We introduce parameter $b_{jk}$, which is set to a sufficiently large constant if capacity type $j$ can produce product type $k$, and to $0$ otherwise. The following constraints ensure that capacity type $j$ is restricted to producing only eligible product types:
 \begin{equation}\label{eq:8}
    Y^{OT}_{ijkt} \leq b_{jk}X^{O}_{ijt},Y^{NT}_{ijkt} \leq b_{jk}X^{N}_{ijt},Y^{NG}_{ijkt} \leq b_{jk}X^{N}_{ijt},\forall i \in \mathcal{I},\forall j \in \mathcal{J},\forall k \in \mathcal{K},\forall t \in \mathcal{T},
 \end{equation}

To account for process differences, the capacity utilization associated with product type $k$ is characterized by distinct capacity utilization rates, denoted as $a^{N}_{jk}$ and $a^{O}_{jk}$ for capacity type $j$ with green and conventional technologies, respectively. 
The maximum number of products that can be produced by a single production line of capacity type $j$ with green and conventional technologies within one decision period is given by $n^{N}_{j}$ and $n^{O}_{j}$, respectively. 
The capacity constraints are thus formulated as follows:
 \begin{equation}\label{eq:9}
    \sum_{\mathcal{K}}a^{O}_{jk}Y^{OT}_{ijkt} \leq n^{O}_{j}X^{O}_{ijt},\sum_{\mathcal{K}}a^{N}_{jk}(Y^{NT}_{ijkt}+Y^{NG}_{ijkt}) \leq n^{N}_{j}X^{N}_{ijt},\forall i \in \mathcal{I},\forall j \in \mathcal{J},\forall t \in \mathcal{T}.
 \end{equation}
Throughout the decision horizon, the manufacturer must ensure that the proportion of green products in its total production meets the predetermined green penetration requirement $\tau$, which is: 
 \begin{equation}\label{eq:10}
    \sum_{\mathcal{I},\mathcal{J},\mathcal{K},\mathcal{T}}Y^{NG}_{ijkt} \geq \tau \sum_{\mathcal{I},\mathcal{J},\mathcal{K},\mathcal{T}}(Y^{OT}_{ijkt}+Y^{NT}_{ijkt}+Y^{NG}_{ijkt}).
 \end{equation}
 
Let $\xi_{kt}$ denote the demand for product $k$ in period $t$. 
If the total production (including both green and traditional variants as defined above) is insufficient to meet the demand, the manufacturer incurs unmet demand $Y^{U}_{kt} \in \mathbb{R}_{+}$ subject to a service level requirement $\lambda$, and a unit penalty cost $c^{U}_{k}$ is imposed for such shortfalls.
The demand balance constraints and service level constraints, which integrate these considerations with the preceding production decisions, are as follows:
 \begin{equation}\label{eq:11}
    \sum_{\mathcal{I},\mathcal{J}}(Y^{OT}_{ijkt}+Y^{NT}_{ijkt}+Y^{NG}_{ijkt})+Y^{U}_{kt} = \xi_{kt},\forall k \in \mathcal{K},\forall t \in \mathcal{T};
 \end{equation}
 \begin{equation}\label{eq:12}
    Y^{U}_{kt} \leq (1-\lambda)\xi_{kt},\forall k \in \mathcal{K},\forall t \in \mathcal{T}.
 \end{equation}

For green product manufacturing, let $\omega_{it}$ denote the peak sunshine hours in the region hosting factory $i$ in period $t$, $E_{i}$ the PV capacity at factory $i$, and $e_{jk}$ the energy consumption per unit of green product.
The next constraint is imposed to ensure that green product production relies exclusively on PV generation: 
 \begin{equation}\label{eq:13}
    \sum_{\mathcal{J},\mathcal{K}}e_{jk}Y^{NG}_{ijkt} \leq E_{i}X^{BR}_{i}\omega_{it},\forall i \in \mathcal{I},\forall t \in \mathcal{T}.
 \end{equation}

Let $\boldsymbol{\xi}$ and $\boldsymbol{\omega}$ denote the vectors of product demand and climate, respectively, and $\boldsymbol{Y}$ denote the vector of all decision variables. The feasible set of $\boldsymbol{Y}$ is defined by constraints \eqref{eq:8}-\eqref{eq:13}, which can be compactly expressed as
$\mathcal{Y}(\boldsymbol{X},\boldsymbol{\xi},\boldsymbol{\omega}) = \left\{\boldsymbol{Y}:\boldsymbol{B}_{\boldsymbol{X}}\boldsymbol{X}\boldsymbol{\omega}+\boldsymbol{B}_{\boldsymbol{Y}}\boldsymbol{Y}+\boldsymbol{B}_{\boldsymbol{\xi}}\boldsymbol{\xi} \geq \boldsymbol{d}, \boldsymbol{Y} \geq \boldsymbol{0}  \right\}$.
Thus, the manufacturer's decision model in the operational stage can be formulated as
 \begin{equation}\label{eq:14}
    \begin{aligned}
        Q(\boldsymbol{X},\boldsymbol{\xi},\boldsymbol{\omega}) = &\min_{\boldsymbol{Y} \in \mathcal{Y}(\boldsymbol{X},\boldsymbol{\xi},\boldsymbol{\omega})} \boldsymbol{C}_{\boldsymbol{Y}}^\top \boldsymbol{Y}\equiv\sum_{\mathcal{I},\mathcal{J},\mathcal{K},\mathcal{T}}c^{O}_{ijk}Y^{OT}_{ijkt} +\sum_{\mathcal{I},\mathcal{J},\mathcal{K},\mathcal{T}}c^{N}_{ijk}(Y^{NT}_{ijkt}+Y^{NG}_{ijkt})+\sum_{\mathcal{K},\mathcal{T}}c^{U}_{k}Y^{U}_{kt}.
    \end{aligned}
 \end{equation}

\subsection{Two-Stage DRO Model and Data-Driven Clustering}
The strategic model in \eqref{eq:7} and the operational model in \eqref{eq:14} can be integrated into a monolithic formulation to jointly determine capacity and production decisions over the planning horizon. 
However, $\boldsymbol{\xi}$ and $\boldsymbol{\omega}$, which represent product demand and renewable energy generation respectively, are random parameters and are not fully known at the time capacity planning decisions are made.
Let $\mathcal{P}$ denote the ambiguity set comprising a family of joint probability distributions of product demands and renewable energy generation, and $\mathbb P$ denote a particular distribution in $\mathcal{P}$. The two-stage green manufacturing capacity planning DRO problem is formulated as the following model:
 \begin{equation}\label{eq:15}
     \min_{\boldsymbol{X} \in \mathcal{X}} \boldsymbol{C}^\top_{\boldsymbol{X}}\boldsymbol{X} + \max_{\mathbb{P} \in \mathcal{P}} \mathbb{E}_{\mathbb{P}} \left[ Q(\boldsymbol{X},\boldsymbol{\xi},\boldsymbol{\omega}) \right].
 \end{equation}
It is well recognized that the choice of ambiguity set $\mathcal P$ is critical in DRO, which determines the trade-off between robustness and performance. An overly large ambiguity set may lead to conservative decisions, while an excessively restrictive one may reduce the DRO model to an optimistic SP model. 
Its construction depends on data availability, quality, and the risk preferences of the decision maker. 
In green manufacturing, several practical challenges must be considered in the development of such ambiguity sets. 

First, there exists a non-trivial statistical correlation between demand uncertainty and the uncertainty associated with renewable energy generation, both of which depend on climate. 
Currently, the majority of existing DRO approaches assume independent marginal distributions or simple parametric correlations to define $\mathcal{P}$.
However, the interdependence of the uncertainties of product demands and renewable energy generation renders these assumptions invalid and necessitates the use of nonparametric or data-driven joint ambiguity sets to capture real-world stochasticity. 
Second, the available data may lack the unified structure necessary for effective analysis and modeling. 
Motivated by a real-world capacity planning problem faced by a food and beverage manufacturer, we observe a pronounced asymmetry in data availability and quality.
On the one hand, long-term, high-frequency historical data on climate factors (e.g., daily peak sunshine hours in the region hosting a factory) are accurate, reliable, and comprehensive. 
Indeed, renewable energy generation is largely determined by peak sunshine hours, making this climatic factor the dominant driver of solar output.
On the other hand, product demand data are often sparse, incomplete, and unreliable. 
This disparity in data availability and quality implies that two data sources cannot be treated equivalently.

In light of the aforementioned issues, together with the central role and abundant data availability of climate factors, we adopt a climate-centric approach to construct the ambiguity set. 
As renewable energy generation is directly driven by key climatic parameters, we let $\boldsymbol{\omega}$ denote these parameters, with slight abuse of notation. 
Specifically, we factorize the joint distribution as $\mathbb{P} = \mathbb{P}_{\boldsymbol{\omega}}(\boldsymbol{\omega}) \cdot \mathbb{P}(\boldsymbol{\xi} \mid \boldsymbol{\omega})$. 
This factorization allows the DRO model in \eqref{eq:15} to be rewritten as:
 \begin{equation}\label{eq:16}
     \min_{\boldsymbol{X} \in \mathcal{X}} \boldsymbol{C}^\top_{\boldsymbol{X}}\boldsymbol{X} + \max_{\mathbb{P}_{\omega} \in \mathcal{P}_{\omega}} \mathbb{E}_{\omega \sim \mathbb{P}_{\omega}} \left[ \max_{\mathbb{P}(\cdot|\omega) \in \mathcal{P}_{\xi|\omega}} \mathbb{E}_{\xi \sim \mathbb{P}(\cdot|\omega)} \left[ Q(\boldsymbol{X},\boldsymbol{\xi},\boldsymbol{\omega}) \right] \right],
 \end{equation}
where $\mathcal{P}_{\boldsymbol{\omega}}$ and $\mathcal{P}_{\boldsymbol{\xi}\mid\boldsymbol{\omega}}$ are the corresponding ambiguity sets. 
This formulation, in its general form, contains nested max-expectation operators, which to the best of our knowledge, has no exact algorithm in the literature to date. 
Nevertheless, this structural decomposition allows us to fully take advantage available data and model renewable energy generation and product demand using different approaches. 

For climate data, we employ the $k$-means clustering method to partition the extensive and high-quality historical climate data set into $S$ clusters with $\mathcal{S} = \{1,2,\cdots,S\}$. 
Each distinct climate cluster $s \in \mathcal{S}$ represents a set of historical periods with statistically homogeneous climate conditions, which primarily determine climate-driven renewable energy generation while preserving its practical relevance for demand uncertainty.  
We estimate the empirical probability $q^{s}$ for cluster $s$ as the ratio of the cardinality of cluster $s$ to the total number of historical samples.
As a result, instead of accounting for continuous variability in climate conditions, we map the range of climatic variability to a finite set of discrete centroids of those clusters, denoted by $\{\omega_1, \omega_2,\cdots, \omega_S\}$.
In practice, an appropriate $S$ should be selected to have a desired trade-off between modeling complexity and descriptive granularity. 
We note that this treatment is well aligned with the abundance and high quality of available climate data.

For product demand data, we aim to characterize the demand distribution for each individual climate cluster.
We assume that product demands are stochastically independent across product types, which is generally consistent with the manufacturer’s operational reality. 
Then, by leveraging available data and industry expertise, we define a box sample space, and lower and upper bounds, denoted by $\boldsymbol{\gamma}^{Ls}$ and $\boldsymbol{\gamma}^{Us}$  respectively, on the first moment of the demand distributions, thereby constructing a typical moment-based ambiguity set in DRO literature.
This treatment, through appropriate adjustment of these bounds, enables us to exploit limited information while providing a conservative characterization of demand uncertainty, which is consistent with the fact that demand data are often incomplete and of relatively lower quality. Mathematically, the ambiguity set for demand in cluster $s$ is represented as
 \begin{equation}\label{eq:17}
     \mathcal{P}^{s} = \left\{ \mathbb{P}^{s}: \boldsymbol{\gamma}^{Ls} \leq \mathbb{E}_{\mathbb{P}^{s}}\left[ \boldsymbol{\xi}^{s} \right] \leq \boldsymbol{\gamma}^{Us}  \right\} \ \mathrm{with} \ \ 
     \Xi^s=\{\boldsymbol{\xi}^{Ls} \leq \boldsymbol{\xi}^{s} \leq \boldsymbol{\xi}^{Us}\}.
\end{equation}
Based on the aforementioned characterization of uncertainty, the model in \eqref{eq:16} can be reduced to
 \begin{equation}\label{eq:18}
     \min_{\boldsymbol{X} \in \mathcal{X}} \boldsymbol{C}^\top_{\boldsymbol{X}}\boldsymbol{X} + \sum_{\mathcal{S}}q^{s}\max_{\mathbb{P}^{s} \in \mathcal{P}^{s}} \mathbb{E}_{\mathbb{P}^{s}} \left[ Q^{s}(\boldsymbol{X},\boldsymbol{\xi}^{s}) \right],
 \end{equation}
where $Q^{s}(\boldsymbol{X},\boldsymbol{\xi}^{s}) = \min_{\boldsymbol{Y} \in \mathcal{Y}^{s}(\boldsymbol{X},\boldsymbol{\xi}^{s})} \boldsymbol{C}_{\boldsymbol{Y}}^\top \boldsymbol{Y}^{s}$ with
 \begin{equation}\label{eq:19}
    \mathcal{Y}^{s}(\boldsymbol{X},\boldsymbol{\xi}^{s}) = \left\{\boldsymbol{Y}^{s}:\boldsymbol{B}^{s}_{\boldsymbol{X}}\boldsymbol{X}+\boldsymbol{B}_{\boldsymbol{Y}}\boldsymbol{Y}^{s}+\boldsymbol{B}_{\boldsymbol{\xi}}\boldsymbol{\xi}^{s} \geq \boldsymbol{d}, \boldsymbol{Y}^{s} \geq \boldsymbol{0}  \right\}.
 \end{equation}
Overall, by reducing the original joint uncertainty characterization to multiple independent ambiguity sets, with each of them corresponding to a climate cluster, as shown in Figure~\ref{fig:Reformulation}, we eliminate the need for a complex structure to describe joint stochasticity and instead introduce parallel clusters. 
The resulting model integrates the core structures of both SP and DRO, balancing SP’s reliance on empirical data for practicality and DRO’s handling of distributional ambiguity for robustness. 
Consequently, we can leverage reliable information about climatic variability and also hedge against less-understood randomness in product demand. 
Note that our treatment is also closely related to \citet{hao_robust_2020} and \citet{wang_featuredriven_2023}, which adopt feature-driven mechanisms to characterize uncertainty in a cluster-wise manner.

\begin{figure}[!ht]
        \centering
    \caption{The Model Reduction with Data-Driven Clustering}
 \includegraphics[trim = 0 100 0 0, scale=.12]{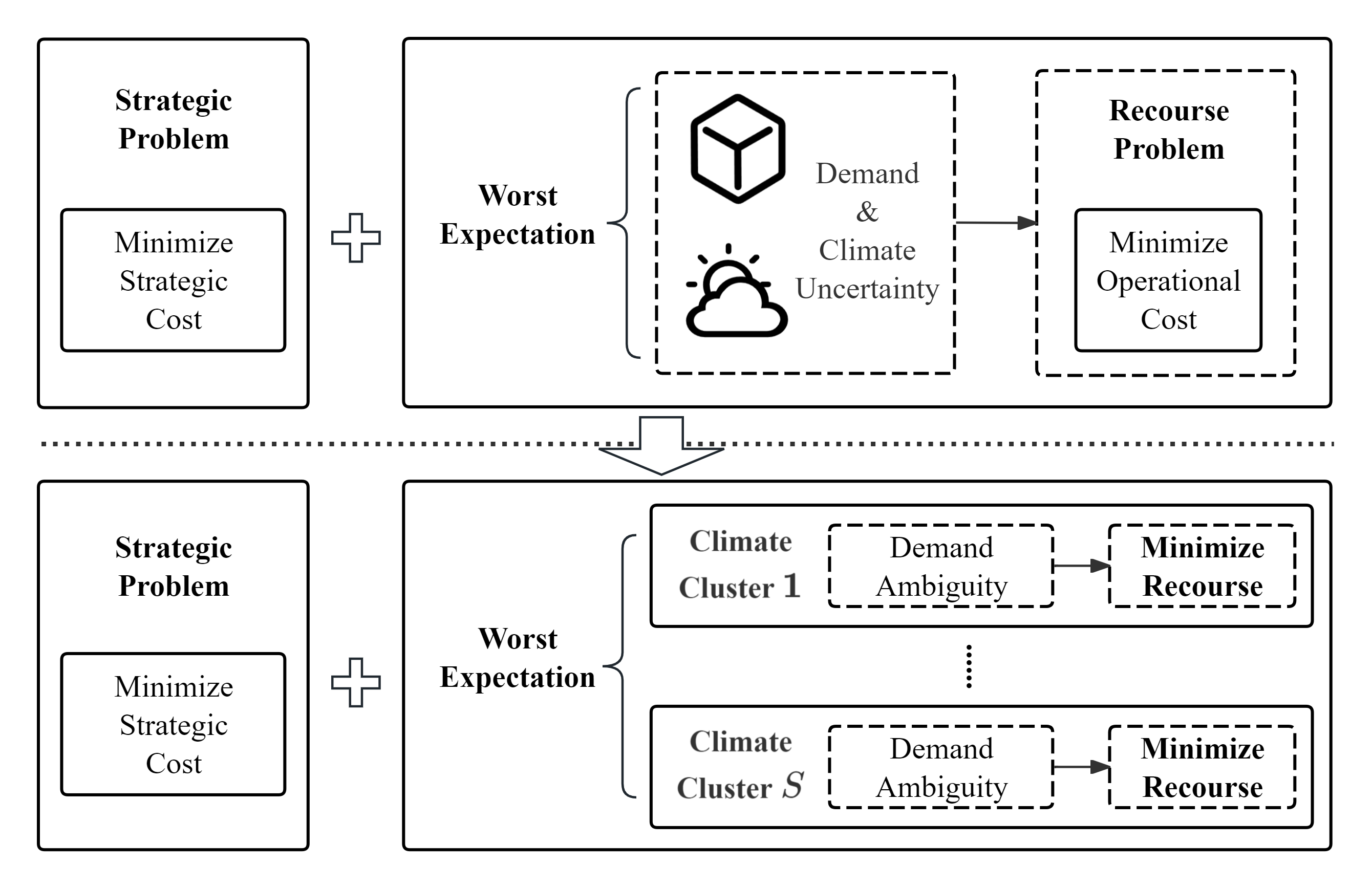}       \label{fig:Reformulation}
\end{figure}

\subsection{Model Reformulations and Structural Properties}
For a given first-stage decision $\boldsymbol{X}^{*}$, each parallel worst-case expected second-stage problem (WESP) corresponds to one climate cluster $s$ and aims to compute the worst-case expected recourse cost under the fixed climate regime of cluster $s$. 
To distinguish this problem from the feasibility check problem introduced later, we refer to WESP as O-WESP in what follows,
 \begin{equation}\label{eq:20}
     \text{O-WESP}:      w^{s}(\mathcal{P}^{s};\boldsymbol{X}^*) :=\max_{\mathbb{P}^{s} \in \mathcal{P}^{s}} \mathbb{E}_{\mathbb{P}^{s}} \left[ Q^{s}(\boldsymbol{X}^{*},\boldsymbol{\xi}^{s}) \right].
 \end{equation}

The next result is rather straightforward. Yet, it implies that a lower bound on the O-WESP can be derived by designing an ambiguity set that satisfies the conditions in \eqref{eq:17}.

\begin{proposition} \label{Proposition:A1}
    Consider any set $\hat{\mathcal{P}}$ such that $\hat{\mathcal{P}} \subseteq \mathcal{P}^{s}$ for cluster $s$, there is $w^{s}(\hat{\mathcal{P}};\boldsymbol X^{*}) \leq w^{s}(\mathcal{P}^{s};\boldsymbol X^{*})$.
\end{proposition}

Following Lemma 3.1 from \citet{shapiro_duality_2001} and Proposition 5 from \citet{lu_two-stage_2024}, we can establish that the optimal distribution for the second-stage problem of model \eqref{eq:18} admits an optimal discrete probability distribution.
This key result allows us to reformulate \eqref{eq:20} into an equivalent yet simpler form based on a finite set of scenarios. 
Let $\Xi^{s*}$ denote the collection of optimal scenarios for the O-WESP corresponding to climate cluster $s$ for given $X^*$. 
Notably, for any given collection $\Xi^{s} \subseteq \Xi^{s*}$, this reformulated problem shares the same structure as a classical column generation master problem (O-CGMP), which enables the application of efficient CG algorithms to accurately solve O-WESP:
 \begin{equation}\label{eq:21}
    \begin{aligned}
        \text{O-CGMP}:\eta^{s}(\Xi^{s}) = \max &\sum_{\boldsymbol{\xi}^{s}\in\Xi^{s}} P_{\boldsymbol{\xi}^{s}}Q^{s}(\boldsymbol{X}^{*},\boldsymbol{\xi}^{s})\\
        \mathrm{s.t.}&\sum_{\boldsymbol{\xi}^{s}\in\Xi^{s}} P_{\boldsymbol{\xi}^{s}} = 1; \ \ P_{\boldsymbol{\xi}^{s}} \geq 0,\forall \boldsymbol{\xi}^{s}\in\Xi^{s};\\
        &\gamma^{Ls}_{kt} \leq \sum_{\boldsymbol{\xi}^{s}\in\Xi^{s}} P_{\boldsymbol{\xi}^{s}}\xi_{kt}^{s} \leq \gamma^{Us}_{kt}, \forall k \in \mathcal{K},\forall t \in \mathcal{T}.\\
    \end{aligned}
 \end{equation}

Let $\alpha^{s}$, $\beta^{Us}_{kt}$ and $\beta^{Ls}_{kt}$ be its dual variables, and $(\alpha^{s*},\boldsymbol{\beta}^{Us*},\boldsymbol{\beta}^{Ls*})$ denote the corresponding shadow prices.
In the CG framework, the pricing problem (PP) is designed to identify new columns that improve the objective value of the O-CGMP.
Specifically, to identify the probability variable with the largest reduced cost, the PP is formulated as a maximization problem that evaluates the potential contribution of each candidate scenario with respect to the optimality of \eqref{eq:21}.
Only scenarios with positive reduced costs are added to the O-CGMP, driving the algorithm toward convergence to the global optimum of \eqref{eq:20}.
 \begin{equation}\label{eq:22}
    \begin{aligned}
        \text{PP}:r^{s*} = \max_{\boldsymbol{\xi}^{s}} & \sum_{\mathcal{K}} \sum_{\mathcal{T}}(\beta^{Ls*}_{kt} - \beta^{Us*}_{kt})\xi_{kt}^{s} - \alpha^{s*} + Q^{s}(\boldsymbol{X}^{*},\boldsymbol{\xi}^{s})\\
        \mathrm{s.t.}&\,\xi^{Ls}_{kt} \leq \xi^{s}_{kt} \leq \xi^{Us}_{kt},\forall k \in \mathcal{K},\forall t \in \mathcal{T}.
    \end{aligned}
 \end{equation}

\begin{proposition} \label{Proposition:A2}
    For a given $\boldsymbol{X}^{*}$, suppose the second-stage recourse function $Q^{s}(\boldsymbol{X}^{*},\boldsymbol{\xi}^{s})$ is finite for $\forall \xi^s\in \Xi^{s*}$. Then, for the problem \eqref{eq:22}, we have one optimal solution that the optimal demand scenario $\xi^{s*}_{kt}$ takes one of the two boundary values, i.e., the lower bound $\xi^{Ls}_{kt}$ or the upper bound $\xi^{Us}_{kt}$  (see its proof in \ref{EC1}). 
\end{proposition}

It is noted that Proposition \ref{Proposition:A2} characterizes the boundary structure of critical scenarios in worst-case distributions, which plays a crucial role in  reducing the complexity of the data encoding/decoding process used in the generative AI approach introduced in the next section.
Furthermore, one can solve the equivalent form of \eqref{eq:22} given in \eqref{eq:ec1} directly using some commercial solvers, or by employing the Karush-Kuhn-Tucker (KKT) conditions to reformulate it into a linearized form (see \ref{EC2} for details).

\begin{proposition} \label{Proposition:A3}
    The second-stage recourse function $Q^{s}(\cdot)$ is a piecewise-linear monotonically non-decreasing function of $\xi_{kt}^{s}$ within the range $[\xi^{Ls}_{kt},\xi^{Us}_{kt}]$, $\forall k \in \mathcal{K}, \forall t \in \mathcal{T}$ (see its proof in \ref{EC1}).
\end{proposition}

Following Propositions \ref{Proposition:A2}–\ref{Proposition:A3}, we next present another equivalent reformulation of the O-WESP problem through enumerating extreme points of $\Xi^s$.   
\begin{corollary} \label{Corollary:A1}
    Denote $\boldsymbol{\mathcal{Z}}:=\{ 0,1 \}^{|\mathcal{K}|\times|\mathcal{T}|}$, the O-WESP problem is equivalent to 
    \begin{equation}
        \begin{aligned}
            \max &\sum_{\boldsymbol{z}\in\boldsymbol{\mathcal{Z}}} P_{\boldsymbol{z}}Q^{s}(\boldsymbol{X}^{*},\boldsymbol{\xi}^{s}(\boldsymbol{z}))\\
            \mathrm{s.t.}&\sum_{\boldsymbol{z}\in\boldsymbol{\mathcal{Z}}} P_{\boldsymbol{z}} = 1; \ \ P_{\boldsymbol{z}} \geq 0,\forall \boldsymbol{z}\in\boldsymbol{\mathcal{Z}}; \  \ \gamma^{Ls}_{kt} \leq \sum_{\boldsymbol{z}\in\boldsymbol{\mathcal{Z}}} P_{\boldsymbol{z}}\xi_{kt}^{s}(\boldsymbol{z}) \leq \gamma^{Us}_{kt}, \forall k \in \mathcal{K},\forall t \in \mathcal{T},\\
        \end{aligned} \nonumber
    \end{equation}
    where $\boldsymbol{\xi}^{s}(\boldsymbol{z})=\boldsymbol{\xi}^{Ls}+\langle\boldsymbol{\xi}^{Us}-\boldsymbol{\xi}^{Ls},\boldsymbol{z}\rangle$.
\end{corollary}

Corollary~\ref{Corollary:A1} shows that O-WESP can, in principle, be formulated as a bilevel max-min linear optimization problem. Through an optimality conditions–based reformulation, this bilevel problem can be converted into a single-level formulation that is theoretically solvable by some commercial solvers. 
However, such a direct solution approach is computationally intractable due to a severe dimensionality explosion. In particular, the cardinality of set $\boldsymbol{\mathcal{Z}}$ grows exponentially with $|\mathcal{K}|\times|\mathcal{T}|$.
Hence, even for moderate-sized values of $|\mathcal{K}|$ and $|\mathcal{T}|$, the resulting formulation involves an extremely large number of lower-level problems. This issue is further exacerbated by the complexity of those lower-level problems, which are real-life capacity planning problems. Hence, solving such an  enumeration-based formulation to optimality is practically infeasible.

To address the challenge of infeasible recourse, a similar reformulation is required.
For a given $\boldsymbol{X}^{*}$ rendering the recourse problem infeasible under some $\boldsymbol{\xi}^{s}$, we have $\mathcal{Y}^{s}(\boldsymbol{X}^{*},\boldsymbol{\xi}^{s}) = \emptyset$.  
By introducing auxiliary variables $\widetilde{\boldsymbol{Y}}^{s}$, we construct a new recourse problem as 
 \begin{equation}
    \begin{aligned}
        \widetilde{Q}^{s}_{f}(\boldsymbol{X}^{*},\boldsymbol{\xi}^{s}) = &\min_{\boldsymbol{Y},\widetilde{\boldsymbol{Y}}^{s} \in \widetilde{\mathcal{Y}}^{s}(\boldsymbol{X}^{*},\boldsymbol{\xi}^{s})} \boldsymbol{1}^{\top} \widetilde{\boldsymbol{Y}}^{s},
    \end{aligned} \nonumber
 \end{equation}
where $\widetilde{\mathcal{Y}}^{s}(\boldsymbol{X}^{*},\boldsymbol{\xi}^{s}) = \left\{(\boldsymbol{Y},\widetilde{\boldsymbol{Y}}^{s}):\boldsymbol{B}^{s}_{\boldsymbol{X}}\boldsymbol{X}^{*}+\boldsymbol{B}_{\boldsymbol{Y}}\boldsymbol{Y}^{s}+\boldsymbol{B}_{\boldsymbol{\xi}}\boldsymbol{\xi}^{s}+\widetilde{\boldsymbol{Y}}^{s} \geq \boldsymbol{d}, \boldsymbol{Y} \geq \boldsymbol{0},\widetilde{\boldsymbol{Y}}^{s} \geq \boldsymbol{0}  \right\}$.

Then, the feasibility check for the worst-case expected problems (F-WESP) can be formulated as
 \begin{equation}
      \text{F-WESP}:\widetilde{w}^{s*} = \max_{\mathbb{P}^{s} \in \mathcal{P}^{s}} \mathbb{E}_{\mathbb{P}^{s}} \left[ \widetilde{Q}^{s}_{f}(\boldsymbol{X}^{*},\boldsymbol{\xi}^{s}) \right]. \nonumber
 \end{equation}

\begin{proposition} \label{Proposition:A4}
    Given a first-stage solution $\boldsymbol{X}^{*}$, $\boldsymbol{X}^{*}$ is feasible for \textup{O-WESP} if and only if the optimal value of \textup{F-WESP} satisfies $\widetilde{w}^{s*}=0$. Moreover, if $\widetilde{w}^{s*}>0$, then $\boldsymbol{X}^{*}$ is infeasible for some $ \boldsymbol{\xi}^{s}\in \Xi^s$ that has a positive probability under a distribution in $\mathcal P^s$ (see its proof in \ref{EC1}).
\end{proposition}

Note that the CG master and pricing problems of F-WESP share the same form as those of O-WESP in \eqref{eq:21} and \eqref{eq:22}, respectively. 
To minimize redundancy, they are provided in \ref{EC2}. 
Next, we present the complete description of the CG algorithm in Algorithm 1 for solving O-WESP. 
It can also be modified slightly to solve F-WESP according to its CG master and pricing problems.\\
\rule{\linewidth}{0.5pt}\vspace{-5pt}
\textbf{Algorithm 1}: CG approach for solving O-WESP\\
\textbf{1.} Given an initial feasible solution set $\Xi^{0}$ and an optimality tolerance $\epsilon$. \\
\textbf{2.} Solve O-CGMP with the set $\Xi^{0}$ and obtain its optimal value $\eta^{s}(\Xi^{0})$. Derive its shadow prices $(\alpha^{s*},\boldsymbol{\beta}^{Us*},\boldsymbol{\beta}^{Ls*})$.\\
\textbf{3.} Solve PP and identify the optimal $\boldsymbol{\xi}^{s*}$. Derive its optimal value $r^{s*}$ as the reduced cost. If $r^{s*} > 0$ go to next step. If $r^{s*} = 0$, go to step 5.\\
\textbf{4.} Add the new column $\boldsymbol{\xi}^{s*}$ into set $\Xi^{0}$ and go to step 2.\\
\textbf{5.} Terminate and let $w^{s*} = \eta^{s}(\Xi^{0})$ to be the optimal value of O-WESP. \vspace{-10pt}\\
\rule{\linewidth}{0.5pt}

\begin{proposition} \label{Proposition:A5}
  Assume $\epsilon=0$. Algorithm 1 terminates in a finite number of iterations, and upon termination  there exists at least one $(k,t)$ such that an optimal solution of O-CGMP satisfies (see its proof in \ref{EC1}):
     \begin{equation}\label{eq:23}
        \sum_{\boldsymbol{\xi}^{s}\in\Xi^{s*}} P_{\boldsymbol{\xi}^{s}}\xi_{kt}^{s} =\gamma^{Us}_{kt}.
     \end{equation}
\end{proposition}

It is worth mentioning that Propositions \ref{Proposition:A1}--\ref{Proposition:A5} collectively reveal structural properties of the worst-case distribution for the O-WESP. 
They serve as a theoretical foundation for two subsequent methodological developments:
(1) the data encoding/decoding process focusing on the boundaries of optimal scenarios, and 
(2) the application of generative AI to enable targeted learning of distributional characteristics and to accelerate the scenario generation process. 
These developments are described in the next two sections.

\section{Solving DRO by An AI-Enhanced Decomposition Algorithm}
As noted earlier, solving general two-stage DRO problems at real-world scale is computationally highly challenging. 
Currently, only a limited number of practical applications have been investigated using two-stage DRO models. 
Although the C\&CG-DRO algorithm \citep{lu_two-stage_2024} demonstrates a strong computational capacity, it can be very time-consuming as an iterative procedure. 
We note that this algorithm provides a flexible framework that allows for the integration of user-designed subroutines (e.g., machine learning or AI-based methods) without compromising solution accuracy. 
Therefore, instead of implementing the standard version, we next present a customized C\&CG-DRO algorithm with a novel generative AI-enhanced scheme that significantly improves our computational efficiency and algorithmic consistency.

\subsection{The Customized C\&CG-DRO Algorithm}
Solving the two-stage DRO problem relies on integrating the basic C\&CG framework \citep{zeng_solving_2013} with Algorithm 1. 
The core idea is to construct a master problem (MP), a relaxation of \eqref{eq:18}, to derive a valid lower bound. 
An upper bound can be derived by solving the O-WESP for a given first-stage decision $\boldsymbol{X}^{*}$. 
Feasibility cuts and optimality cuts are then generated using the scenarios produced by Algorithm 1. Let $\Xi^{s(o)}$ and $\Xi^{s(f)}$ denote the collections of scenarios associated with optimality cuts and feasibility cuts, respectively, that are generated for cluster $s$ when solving O-WESP and F-WESP. The MP is given by
 \begin{subequations}
    \begin{align}
        \text{MP}:\underline{g} = \min &\, \boldsymbol{C}_{\boldsymbol{X}}^\top \boldsymbol{X} + \sum_{\mathcal{S}} q_{s}\eta_{s} \label{eq:24a}\\
        \mathrm{s.t.}&\, \boldsymbol{X} \in \mathcal{X};\label{eq:24b}\\
        &\eta_{s} \geq \max \sum_{\boldsymbol{\xi}^{s}\in\Xi^{s(o)}} P^{(o)}_{\boldsymbol{\xi}^{s}}Q^{s}(\boldsymbol{X},\boldsymbol{\xi}^{s}),\forall s \in \mathcal{S}; \ \  \boldsymbol{Y}^{s}_{\boldsymbol{\xi}^{s}} \in \mathcal{Y}^{s}(\boldsymbol{X},\boldsymbol{\xi}^{s}),\forall s \in \mathcal{S},\forall \boldsymbol{\xi}^{s} \in \Xi^{s(o)};\label{eq:24c}\\
        &0 \geq \max \sum_{\boldsymbol{\xi}^{s}\in\Xi^{s(f)}} P^{(f)}_{\boldsymbol{\xi}^{s}}\widetilde{Q}_{f}^{s}(\boldsymbol{X},\boldsymbol{\xi}^{s}),\forall s \in \mathcal{S}; \ \ 
     \boldsymbol{Y}^{s}_{\boldsymbol{\xi}^{s}},\widetilde{\boldsymbol{Y}}^{s}_{\boldsymbol{\xi}^{s}} \in \widetilde{\mathcal{Y}}^{s}(\boldsymbol{X},\boldsymbol{\xi}^{s}),\forall s \in \mathcal{S},\forall \boldsymbol{\xi}^{s} \in \Xi^{s(f)},\label{eq:24d}
    \end{align}
 \end{subequations}
where the two embedded maximization problems are LPs, with $P^{(o)}_{\xi^s}$ and $P^{(f)}_{\xi^s}$ as the decision variables. 
MP can be equivalently reformulated into a mixed-integer linear program, shown in \ref{EC2}.

If MP is infeasible, then the DRO problem \eqref{eq:18} is also infeasible, referred to as DRO infeasibility. 
When implementing C\&CG-DRO, it is also worth mentioning that the independence among the scenario sets $\Xi^{s(o)}$ (and $\Xi^{s(f)}$) naturally enables parallel computation of WESP problems.  
Next, we present an algorithmic procedure for solving the full problem in Algorithm 2, where the upper and lower bounds of MP are denoted as $UB$ and $LB$ respectively, and the tolerance is denoted as $TOL$.\\
\rule{\linewidth}{0.5pt}\vspace{-5pt}
\textbf{Algorithm 2}: C\&CG-DRO algorithm\\
\textbf{1.} Set $UB = +\infty$, $LB = -\infty$, $TOL = 10^{-5}$ and $\Xi^{s(o)} = \emptyset$, $\Xi^{s(f)} = \emptyset$ for cluster $s$, $s\in \mathcal S$.\\
\textbf{2.} Solve MP. If MP is infeasible, terminate and report the two-stage DRO problem is infeasible. If it is unbounded, get a feasible solution $\boldsymbol{X}^{*}$. Otherwise, get its optimal solution $\boldsymbol{X}^{*}$ and value $\underline{g}$, let $LB = \underline{g}$.\\
\textbf{3.} For every cluster $s\in \mathcal S$, solve F-WESP with its optimal value $\widetilde{w}^{s*}$. \\
\hspace*{2em} \textbf{(3-a)} If $\widetilde{w}^{s*} > 0$, set $w^{s*} = +\infty$. Derive all scenarios and their associated probabilities, let $\Xi^{s(f)*}$ be the set of scenarios that the associated probabilities are greater than $0$. Update  $\Xi^{s(f)} = \Xi^{s(f)*}\cup\Xi^{s(f)}$.\\
\hspace*{2em} \textbf{(3-b)} If $\widetilde{w}^{s*} = 0$, solve O-WESP for cluster $s$, with its optimal value $w^{s*}$ and $\Xi^{s(o)*}$, the set of scenarios that the associated probabilities are greater than $0$. Update $\Xi^{s(o)} = \Xi^{s(o)*}\cup\Xi^{s(o)}$.\\
\textbf{4.} With updated $\Xi^{s(o)}$ and  $\Xi^{s(o)}$ for $s\in \mathcal S$, add the corresponding constraints and variables into MP. Set $UB = \min \left \{ UB,\boldsymbol{C}_{\boldsymbol{X}}^\top \boldsymbol{X}^{*} + \sum_{\mathcal{S}} q_{s}w^{s*} \right \}$.\\
\textbf{5.} If $ UB-LB\leq TOL$, terminate and output optimal solution $\boldsymbol{X}^{*}$. Otherwise go to step 2. \vspace{-10pt}\\
\rule{\linewidth}{0.5pt}

We note that the C\&CG-DRO algorithm exhibits a nested structure.
As shown in Step~3, the solution of the F-WESP and O-WESP relies on the CG procedures presented in Algorithm 1. 
It has been established that effective initialization of columns can substantially reduce the computational time required by Algorithm 1.  
One straightforward approach is to reuse the scenarios in the sets $\Xi^{s(f)}$ and $\Xi^{s(o)}$ obtained from a previous iteration as the initial columns. 
Moreover, by the construction of the O-CGMP and F-CGMP, non-optimal scenarios are assigned zero probability upon termination of Algorithm 1. It suggests that initializing Algorithm 1 with any set of scenarios, provided they belong to the underlying sample space, does not affect its convergence and therefore preserves the correctness of the overall C\&CG-DRO algorithm.

\subsection{An AI-Enhanced C\&CG-DRO Through GANs}
Among various AI approaches, the generative adversarial network (GAN) has emerged as a powerful framework for learning complex data patterns and generating high-quality synthetic samples. 
Standard GANs are known to suffer from training instabilities, including mode collapse and sensitivity to hyperparameter selection. 
To overcome these limitations, the Wasserstein generative adversarial network with gradient penalty (WGAN-GP) has been proposed as an improved variant of GAN. 
By enforcing a Lipschitz constraint through a gradient penalty, WGAN-GP significantly improves training stability and convergence behavior \citep{gulrajani_improved_2017}.

In our study, we adopt a conditional WGAN-GP framework to learn the mapping between first-stage decisions and the supporting scenarios of the worst-case distribution.
Specifically, a conditional architecture is introduced to incorporate the evolving first-stage decisions and related parameters as conditioning information, which allows the generator to produce scenarios tailored to each iteration of the C\&CG-DRO algorithm. 
The WGAN-GP configuration is employed for training to ensure training stability and to accurately capture the distributional structure of worst-case scenarios.
Given the iterative nature of the C\&CG-DRO process, under which scenario distributions change dynamically across iterations, this conditional WGAN-GP design enables adaptive and consistent scenario generation throughout the algorithm, thereby enhancing both computational efficiency and solution quality.

Note that, conventionally, GANs (including variants) are mainly utilized in image and video generation, where high-quality images are produced through adversarial training between a generator and to critic. 
The generator maps a fixed-dimensional random noise vector and an observed feature vector into structured outputs designed to mimic the statistical properties of real target data. 
Conversely, the critic takes real or synthetic outputs and observed features as input and learns to distinguish generated samples from real ones.
Nevertheless, our scenario-related data lack the spatial and structural connections that are typical of images and videos. 
As a result, directly applying GANs to such data is ineffective. 
To address this issue, we develop a data encoding/decoding scheme to transform scenario-related information into structured binary images, which facilitates effective learning of WGAN-GP. 
The decoding process then converts the generated images into a batch of scenarios compatible with the C\&CG-DRO framework.

\begin{figure}[!ht]
        \centering
    \caption{\textbf{Workflow of One Iteration of AI-Enhanced C\&CG-DRO}}
 \includegraphics[trim = 0 170 0 0, scale=.10]{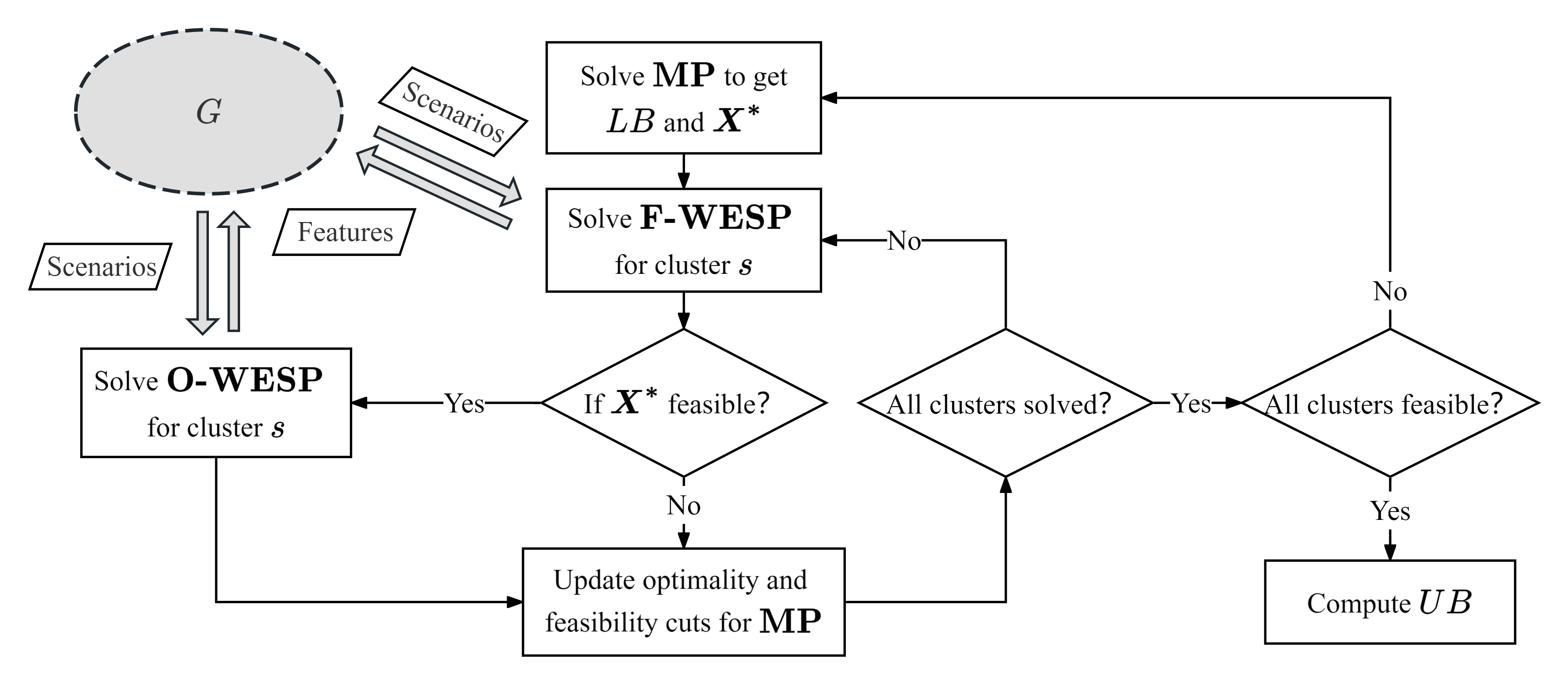}       \label{fig:Algo_Architecture}
\end{figure}

Figure \ref{fig:Algo_Architecture} presents the workflow of one iteration of AI-enhanced C\&CG-DRO. As shown in this figure, the integrated generator (represented by the shaded oval labeled $G$) closely interacts with the solution procedures, i.e., the CG algorithms of F-WESP or O-WESP. 
In each iteration, it receives feature vectors from F-WESP or O-WESP and generates a set of scenarios, which are then used to initialize the corresponding CG algorithms. 
Here, the feature vector is constructed by concatenating the current first-stage solution $\boldsymbol{X}^{*}$ and the model parameters, including both left-hand-side and right-hand-side coefficients of the optimization model.
This initialization accelerates the CG procedures and improves the overall computational performance of the C\&CG-DRO algorithm.
In the following, we present the detailed descriptions of the generator and critic architectures. The design of the encoding/decoding scheme and the details of AI training, which are largely independent of the C\&CG-DRO, are deferred to the next section.

\noindent $\blacksquare$ \ \textbf{Generator}. 
The generator is designed to synthesize two-dimensional images, conditioned on both noise and feature inputs.
It receives two distinct input vectors: a noise vector sampled from a latent distribution and a feature vector that contains auxiliary conditioning information. 
Each vector is independently projected into a high-dimensional space via dedicated fully connected layers, resulting in a set of spatial maps.
These maps are concatenated along the channel dimension to form a composite input tensor, enabling the integration of stochastic variability with domain-specific feature information.

The concatenated tensor is processed by a lightweight convolutional neural network (CNN), a widely adopted architecture to extract critical spatial patterns and reconstruct structured outputs. 
It applies a series of 2D convolutional layers with LeakyReLU activations, which alleviate the vanishing-gradient issue commonly observed in deep networks.  
These layers progressively transform the input into a single-channel image.
A sigmoid activation function is used at the final layer, which ensures smooth gradient propagation during backpropagation. 
After the final activation layer, a custom straight-through binarization module is applied. Once the generator is trained, this module maps the outputs to binary images with 0-1 values.

\begin{figure}[!ht]
        \centering
    \caption{The Architecture of the Generator}
 \includegraphics[trim = 0 250 0 0, scale=.11]{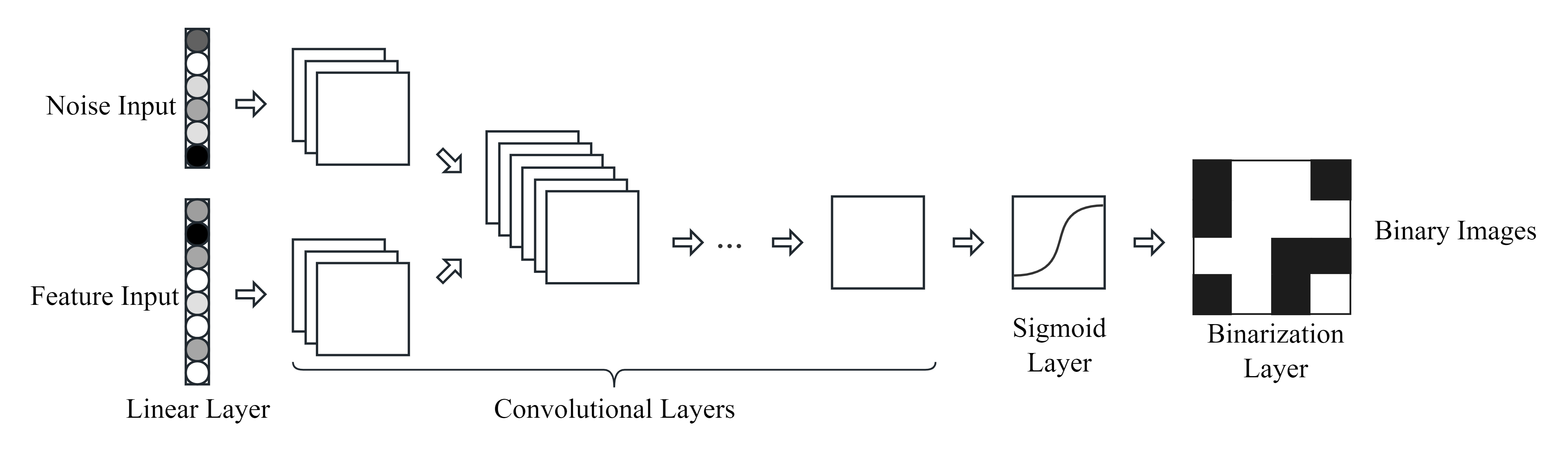}       \label{fig:Generator}
\end{figure}

\noindent $\blacksquare$ \ \textbf{Critic}. 
The critic functions as a conditional evaluator that assesses the authenticity of generated or real images in conjunction with auxiliary feature vectors. 
A feature vector is first transformed into a spatial map through a fully connected layer and reshaped to align with the spatial dimensions of the input image.
This map is then concatenated with the image along the channel dimension to form the input tensor to a CNN. 

To ensure accurate evaluation, this CNN is designed to be more expressive than that employed in the generator. 
It consists of a sequence of convolutional layers with progressively increasing and then decreasing channel depths, enabling multi-scale feature extraction. 
Similarly, LeakyReLU activations are used throughout to promote training stability and improve learning performance. 
The final CNN output is flattened and passed through a fully connected layer to produce a scalar value representing the Wasserstein critic score, which is used for training under the WGAN-GP framework.
\begin{figure}[!ht]
        \centering
    \caption{The Architecture of the Critic}
 \includegraphics[trim = 0 250 0 0, scale=.11]{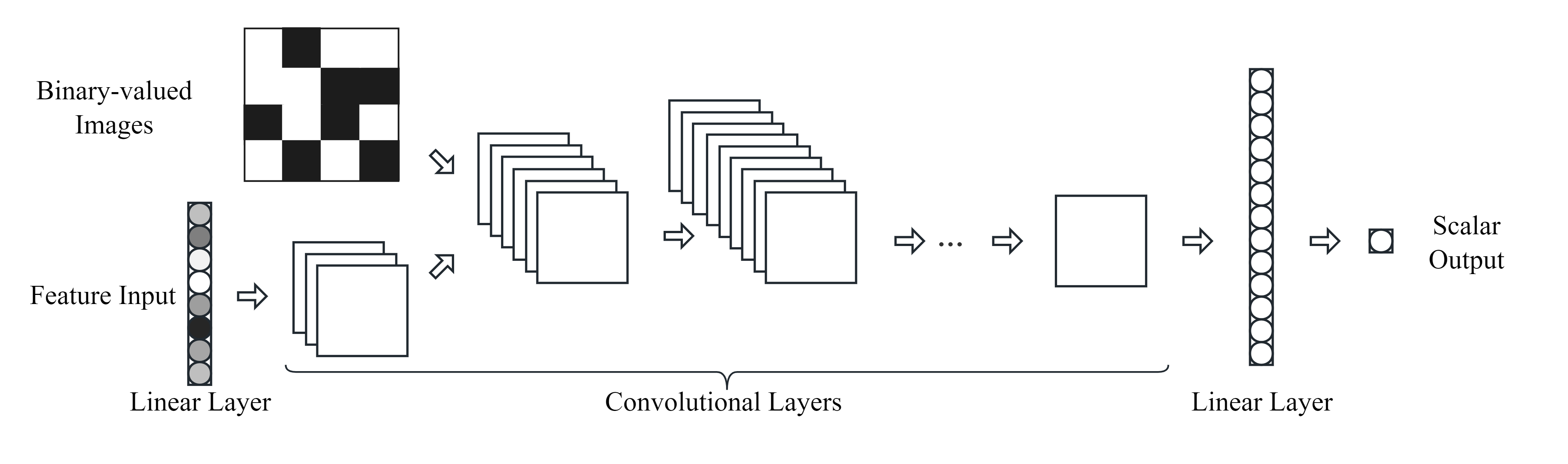}       \label{fig:Critic}
\end{figure}

In the prediction stage, the outputs of the WGAN-GP are decoded into a set of scenarios, which can be directly incorporated into the CG algorithm of F-WESP/O-WESP to enhance C\&CG-DRO while preserving its exactness. 
Moreover, by solving a straightforward LP for a given $\boldsymbol{X}^{*}$, we can efficiently compute a lower bound for O-WESP/F-WESP. 
If those scenarios satisfy condition \eqref{eq:23} of Proposition~\ref{Proposition:A4}, this lower bound is likely to be nontrivial. 
This, in turn, provides a high-quality predictor for the cost/feasibility of $\boldsymbol{X}^{*}$. 
Consequently, instead of executing a full call to Algorithm~1, we can alternatively leverage this prediction step and directly augment MP with the generated scenarios to obtain an improved solution $\boldsymbol{X}^*$. 
Hence, this WGAN-GP module can not only improve CG's initialization, but also serve as a robust surrogate that can replace this computationally intensive procedure. 
This dual functionality highlights the effectiveness of integrating AI techniques to improve both computational efficiency and solution quality of complex optimization algorithms such as C\&CG-DRO.

\section{Data Encoding/Decoding and Training for Generative AI}
As previously noted, we have developed a novel encoding/decoding scheme to support GANs in handling scenario-related information. 
In particular, our approach capitalizes on the structural properties derived from the original DRO model. 
This enable us to alleviate the black-box nature of AI tools, which often disregard explicit boundary conditions, and ensure that the predicted scenarios exhibit the desired optimal structural properties, thereby substantially enhancing prediction accuracy.
Moreover, through decoding, the WGAN-GP generates batches of scenarios consistent with the worst-case distribution, thereby further enhancing computational efficiency.

\subsection{Data Encoding/Decoding}
Our data encoding/decoding scheme is tailored to capture scenarios that are critical to the worst-case distribution for a given feature vector.
As noted in Proposition \ref{Proposition:A2}, a scenario generated by PP has the property that each uncertain variable attains either its upper or lower bound. 
Accordingly, our encoding scheme is designed to convert such a scenario into a one-dimensional binary (label) vector, with 0 denoting the lower bound and 1 denoting the upper bound of its sample space.
Instead of applying WGAN-GP directly to individual binary vectors, we seek to fully leverage its strong image-processing capability. To this end, we introduce a sampling operation on the resulting binary vectors to generate multiple such vectors, which are then sorted to form a binary matrix, i.e., a binary image.
Irrespective of the original probabilities assigned to these scenarios in Algorithm 1, we perform uniform sampling over the aforementioned binary vectors to obtain a collection of scenarios spanning a wide range, thereby promoting diversity in the generative process. As for sorting, since any binary vector can be interpreted as the binary representation of an integer, we can simply sort the binary vectors by  their corresponding integer values.

Note that, for any scenario set generated by Algorithm~1 for a given $\boldsymbol{X}^*$, we can construct a binary matrix (i.e., an image) of the same dimensions by applying the aforementioned operations. Hence, solving one DRO instance via C\&CG-DRO generally yields multiple consistent images that are compatible with neural network training.   
The complete encoding process is illustrated in Figure \ref{fig:Sampling}. Regarding decoding, once a binary image is generated by the generator, each of its rows can be readily decoded, with reference to the predefined sample space, to construct a feasible scenario. A complete image yields a set of high-quality scenarios underlying the worst-case distribution for given $\boldsymbol{X}^*$.  

\begin{figure}[!ht]
        \centering
    \caption{The Data Encoding Process}
 \includegraphics[trim = 0 230 0 0, scale=.105]{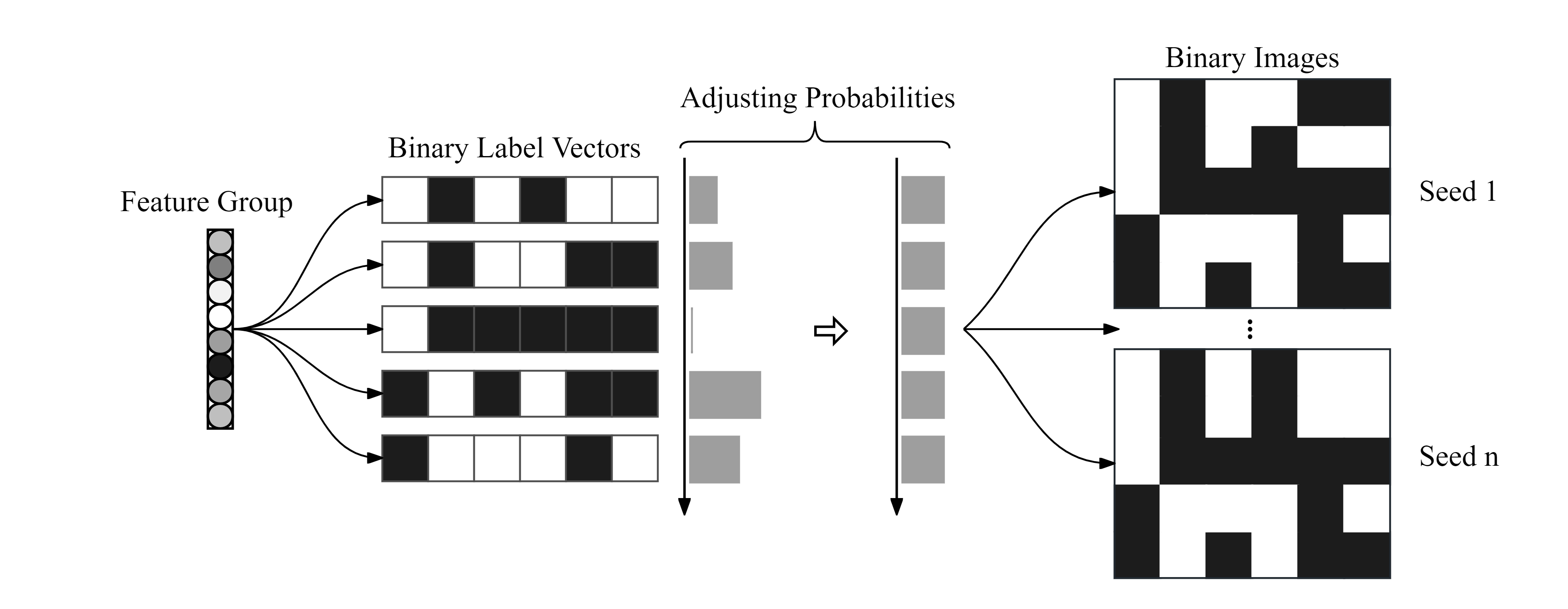}       \label{fig:Sampling}
\end{figure}

To enhance representational diversity, we generate multiple image samples from each feature vector by varying the random seeds used in sampling, thus producing a set of distinct binary image matrices. 
This strategy not only effectively expands the training dataset but also facilitates a more thorough exploration of the underlying probabilistic structure embedded in the original data. 
In practical implementations, the sampling probabilities can be flexibly adjusted according to different strategies.
For instance, when the generator is trained to provide a tighter lower bound, preserving the original scenario probabilities during sampling can substantially improve training effectiveness towards a more accurate estimate.  

\subsection{Network Training}
In parallel, we normalize the feature vectors and apply principal component analysis (PCA) to reduce computational complexity while preserving 99\% of the total variance. 
The resulting dimension-reduced feature vectors are combined with the data encoding output, specifically the sorted binary matrices representing structured ``signal images,'' to serve as the inputs for the downstream generative modeling framework.
For training, we construct a dataset by solving a collection of small-scale instances sampled from a pre-defined parameter space.
The detailed dataset preparation and training hyperparameter settings are provided in \ref{EC3}.

\begin{figure}[!ht]
        \centering
    \caption{Training Loss and Performance}
 \includegraphics[trim = 0 50 0 0, scale=.25]{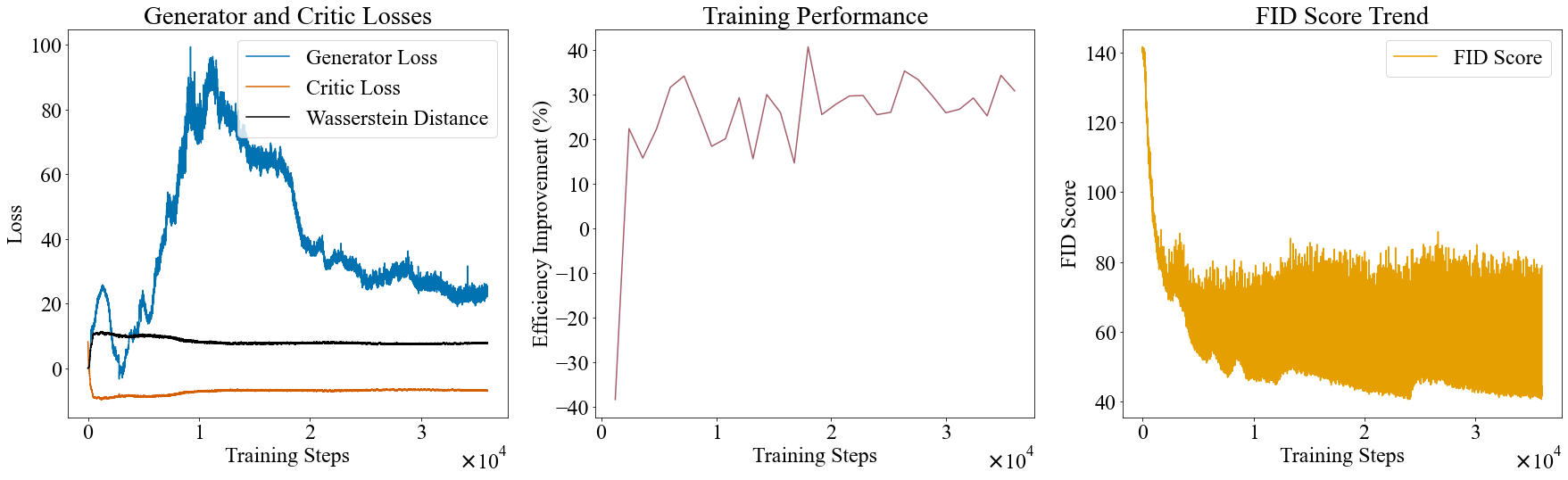}       \label{fig:Loss}
\end{figure}

The training loss curve, the efficiency improvement, and the quality of generated samples are shown in Figure~\ref{fig:Loss}. 
The efficiency improvement is measured by the relative solution time reduction when the samples (i.e., the set of scenarios) produced by the generator of different training steps are embedded into the C\&CG-DRO, compared to the standard C\&CG-DRO implementation.
The quality of generated samples is measured by the divergence with respect to real samples using the Fréchet inception distance (FID).
Similar to the observation made in \citet{gulrajani_improved_2017}, WGAN-GP does not exhibit a clear or monotonic trend in its loss components, unlike traditional GANs. However, as training progresses, adversarial dynamics appear in the generator loss, 
accompanied by the steady efficiency improvement.
Also, the FID score of the generated samples gradually decreases, indicating better sample quality. The oscillations in FID during the later steps of training indicate that the generated instances avoid overfitting while maintaining sufficient diversity.

\begin{figure}[!ht]
        \centering
    \caption{Generated Fake Images and Real Images}
 \includegraphics[trim = 0 70 0 0, scale=.32]{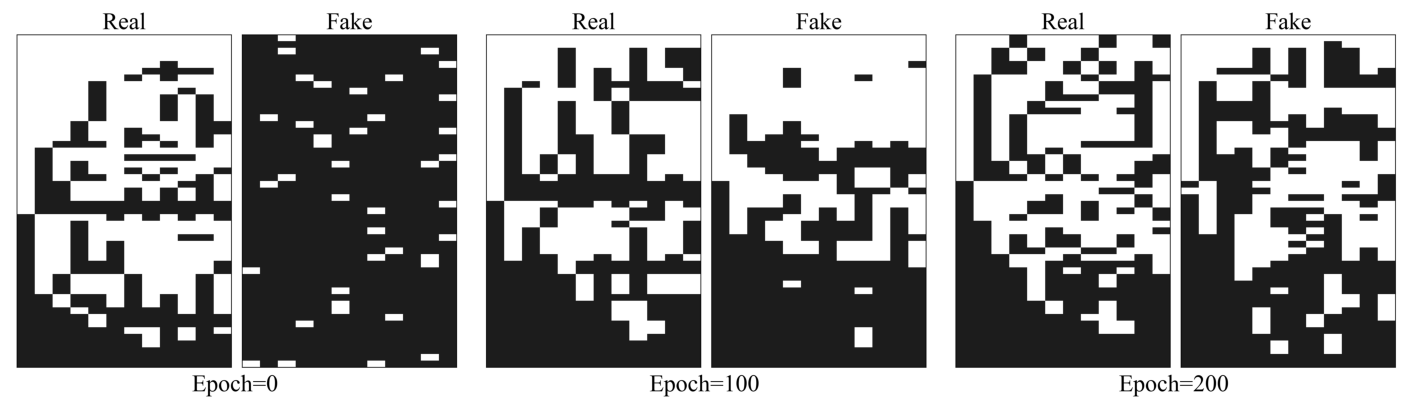}       \label{fig:Image}
\end{figure}

To provide an intuitive understanding of the quality of the generated samples, images produced by the generator over the first 200 training epochs are illustrated in Figure~\ref{fig:Image}. 
The differences between real and generated binary images decrease substantially as training progresses. 
This observation demonstrates not only the effectiveness of our encoding/decoding scheme in representing critical and latent structural patterns, but also the ability of WGAN-GP to capture and reproduce these patterns. In particular, WGAN-GP learns a high-quality representation of the latent mapping from a feature vector to its corresponding binary image (i.e., the associated scenario set underlying the worst-case distribution). 
Consequently, rather than relying on the standard iterative CG procedure that generates a single scenario per iteration, this learned mapping acts as a fast and effective surrogate that produces a collection of probabilistically significant critical scenarios.

Prior studies have largely overlooked optimization-aware encoding/decoding schemes, which limits the integration of AI models into optimization algorithms. 
To the best of our knowledge, this work is the first to explicitly embed structural properties of the underlying model directly into the learning process through the design of a context-aware encoding/decoding scheme, thereby ensuring the feasibility and quality of the predicted solutions.
Moreover, the image-based encoding/decoding scheme enables efficient batch prediction. 
The resulting batch of generated scenarios provides a broader representation of the underlying probability distribution than scenarios generated sequentially based on local information, offering a significant computational advantage. 
The effectiveness of the proposed encoding/decoding scheme and the trained generator will be demonstrated on real-world computational instances in the subsequent section.

\section{Experiments and Analyses}
In this section, a set of experimental studies is conducted with respect to instances generated from real-world data provided by our industry partner. 
Our experiments primarily focus on three aspects: (i) the computational properties of the AI-enhanced C\&CG-DRO algorithm, (ii) the performance and sensitivity of the proposed green manufacturing capacity planning model under realistic data settings, and (iii) a comparison between the DRO and the widely adopted SP approach in capacity planning.
All experiments are implemented in Python on a server with an AMD 5600X processor, 16 GB of RAM, and Gurobi 10.1 as the solver.

\subsection{Data Set and Computational Instances}
\noindent\textbf{Industrial Context and Critical Parameters:}
Our industry partner, a major food and beverage manufacturer in southwest China, operates three factories, each located in a different province. 
Referring to these provinces as regions~1,~2, and~3 respectively, we label the corresponding factories as factory~1,~2, and~3. The initial capacity configurations are summarized in Table~\ref{table:capacity_configuration}, where capacities are measured by the number of production lines installed with conventional technology. 
Capacity types~I and~III are able to produce product families~A,~B, and~C, whereas capacity type~II is restricted to producing product family~C. 
Currently, this manufacturer is considering a production plan over a horizon consisting of $T=4$ decision periods, with each period corresponding to one fiscal quarter. 
Over this horizon, the decision maker is expected to achieve a minimum green penetration requirement of 10\% by switching from conventional to green technologies and investing in renewable generation facilities, while also maintaining a product service-level of at least 99\% in every period. 
In addition, capacity adjustments, whether expansions or reductions, are limited to at most one production line per period.

Next, we present relevant model parameters, noting that all cost values are expressed in units of $10^6$ unless otherwise stated. 
The unit cost of capacity expansion for the three capacity types is $55.00$, $14.00$ and $42.00$ respectively. 
Given that capacity reduction can recover a certain amount of residual value, the unit cost of capacity reduction for the three capacity types is $-8.00$, $0.20$ and $-5.40$, respectively. 
The unit cost of technology upgrade for the three capacity types is $5.50$, $2.10$ and $8.00$, respectively. 
The three factories have different capacities for installing distributed PV (kW), specifically $4000$, $2500$ and $3000$, and consequently they have different fixed investment costs, which are $14.00$, $8.75$ and $10.50$ respectively.
For production lines with conventional technology, the corresponding capacities are $449.97$, $97.85$ and $262.08$, while they are $440.30$, $95.89$ and $256.05$ respectively for production lines with green technology.
The capacity utilization rates, the costs per unit product (not expressed in units of $10^6$) and the electricity consumption (kWh) per unit product produced on conventional and green technology production lines are shown in Table \ref{table:production_parameters} in the form of ``capacity utilization rate/unit production cost/unit electricity consumption'', where ``$-/-/-$'' indicates not applicable. 
In the food and beverage sector, the opportunity cost of unmet demand caused by supply shortages is relatively low, since consumers face abundant substitutes and exhibit low switching costs.
Accordingly, the penalty costs for unit unmet demand for the three product families, which are relatively lower than their unit production cost, are $0.15,0.18$ and $0.12$, respectively (not expressed in units of $10^6$).

\begin{table}[!ht]
\caption{Initial Configuration of Capacities}
\centering
        {\small\def\arraystretch{1.5} 
\begin{tabular*}{0.5\textwidth}{@{\extracolsep{\fill}}lccc}
\hline
\hline
Factory & Capacity I & Capacity II & Capacity III\\
\hline
Factory 1 & 3 & 1 & 2\\
Factory 2 & 2 & 0 & 0\\
Factory 3 & 0 & 0 & 2\\
\hline
\hline
\end{tabular*}
\label{table:capacity_configuration}
}
\end{table}

\begin{table}[!ht]
\caption{Production Related Parameters}
\centering
        {\small\def\arraystretch{1.5} 
\begin{tabular*}{0.7\textwidth}{@{\extracolsep{\fill}}lccccccc}
\hline
\hline
Technology & Product & Capacity I & Capacity II & Capacity III\\
\hline
\multirow{3}{*}{Old} & A & 1.00/1.86/0.34 & -/-/- & 1.13/2.47/0.39\\
~ & B & 1.00/1.50/0.37 & -/-/- & 1.00/1.66/0.37\\
~ & C & 1.04/1.33/0.39 & 1.00/1.50/0.29 & 1.04/1.93/0.39\\
\multirow{3}{*}{Green} & A & 1.01/1.91/0.34 & -/-/- & 1.13/2.53/0.39\\
~ & B & 1.01/1.53/0.37 & -/-/- & 1.00/1.69/0.37\\
~ & C & 1.04/1.03/0.39 & 1.01/1.54/0.29 & 1.04/1.97/0.39\\
\hline
\hline
\end{tabular*}
\label{table:production_parameters}
}
\end{table}

\noindent\textbf{Data Preparation and Computational Instances:}
Historical climate data of quarterly peak sunshine hours from 1992 to 2022 has been collected for the three regions, as shown in \ref{EC3}. 
It is observed that region~3 exhibits the most favorable solar irradiance conditions. 
To capture additional patterns in climate variation across those three regions, both mean and range features are incorporated into the historical data. 
Subsequently, $k$-means clustering is applied to this augmented climate dataset to support the construction of ambiguity sets.  
The procedure for setting the sample space and the first moment bounds is detailed in \ref{EC3}.

As noted in Section 3, identifying an appropriate number of clusters is essential for achieving computational efficiency, even though using a larger number of clusters may help us better capture the underlying randomness. 
Accordingly, we conducted a series of experiments on the base case instance, which is constructed using the aforementioned parameters, with different numbers of clusters. 
Computational results are reported in Table \ref{table:cluster_number}. 
We observe that all performance metrics stabilize when the number of clusters reached 10, indicating that the clustered data points can adequately represent the original climate variability. 
Based on this observation, all our experimental instances are constructed with $S=10$ clusters. 

\begin{table}[h]
\caption{Results of Different Numbers of Clusters}
\centering
        {\small\def\arraystretch{1.5} 
\begin{tabular*}{0.85\textwidth}{@{\extracolsep{\fill}}lcccccc}
\hline
\hline
Number of Clusters & $S=2$ & $S=4$& $S=6$& $S=8$& $S=10$& $S=12$\\
\hline
Total Cost                 & 228.45 & 256.56 & 283.05 & 294.509 & 389.75 & 389.55\\
Strategic Cost             & 73.25  & 101.25 & 128.25 & 139.65  & 235.65 & 235.65\\
\quad  -Capacity Adjustment Cost      & 43.00  & 71.00  & 98.00  & 109.40  & 205.40 & 205.40\\
\quad  -Tech-upgrade Cost             & 11.00  & 11.00  & 11.00  & 11.00   & 11.00  & 11.00\\
\quad  -Renewable Power Investment    & 19.25  & 19.25  & 19.25  & 19.25   & 19.25  & 19.25\\
Tactical Cost                & 155.20 & 155.31 & 154.80 & 154.86  & 154.10 & 153.90\\
Avg. Green Penetration Rate (\%)       & 10.64  & 10.61  & 10.99  & 10.53   & 11.03  & 11.06\\
Avg. Service Level (\%)                & 99.00  & 99.00  & 99.00  & 99.00   & 99.00  & 99.00\\
\hline
\hline
\end{tabular*}
\label{table:cluster_number}
}
\end{table}

\subsection{Algorithm Performance Comparison: Advantages of AI Enhancements}
Across the 100 randomly generated testing instances described in \ref{EC3}, we compare the proposed AI-enhanced C\&CG-DRO algorithm with basic C\&CG, which is a widely applied in the literature \citep{lu_review_2021}, and the standard C\&CG-DRO algorithm. 
To have a fair and consistent comparison, scenarios produced by the generator are employed solely for initializing the CG procedure and are not directly used to augment the master problem.   
All methods are implemented with an optimality tolerance of $10^{-5}$ and a time limit of 3,000 seconds. Since existing implementations of the basic C\&CG for DRO cannot handle infeasible recourse problems, we extend it by incorporating the feasibility-checking procedure and the associated feasibility cuts described in Section~3.2, thereby ensuring a fair and meaningful comparison with the other two methods.
The overall computational performance of these methods is illustrated in Figure~\ref{fig:box_CPU_time}, the computational time spent on individual algorithmic components is presented in Figure~\ref{fig:CPU_time}, and summary statistics of the primary operations are provided in Table~\ref{table:algorithm_performance}, all averaged over 100 testing instances.

As illustrated in Figure~\ref{fig:box_CPU_time}, both C\&CG-DRO and the AI-enhanced C\&CG-DRO generally achieve an order-of-magnitude speedup compared to basic C\&CG. 
Moreover, these two approaches exhibit substantially more stable computational behavior. In contrast, the basic C\&CG shows huge variability in solution time. 
Between the AI-enhanced C\&CG-DRO and C\&CG-DRO, the former consistently  outperforms the latter by requiring less computational time, exhibiting reduced variability, and producing fewer outliers. 
Such consistent performance is critical for practical decision-making, as management teams often need to evaluate the model under multiple parameter settings, for example, in sensitivity analyses.

\begin{figure}[!ht]
    \centering
    \begin{minipage}{.49\textwidth}
        \centering
    \caption{Overall Algorithm Performance}
\includegraphics[trim = 0 50 0 0, scale=.345]{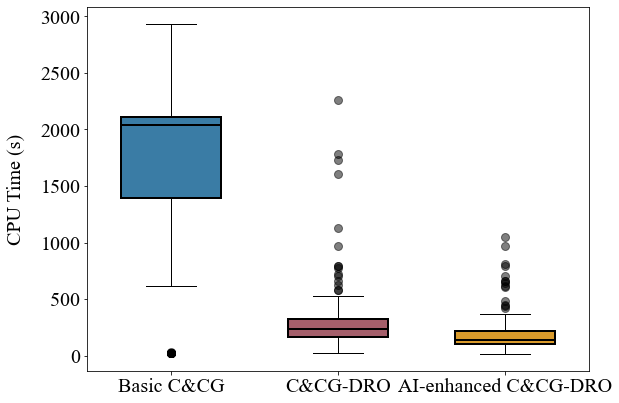} \label{fig:box_CPU_time}
     \end{minipage}%
    \begin{minipage}{.49\textwidth}
        \centering
    \caption{Performance of Algorithmic Components}
\includegraphics[trim = 0 50 0 0, scale=.345]{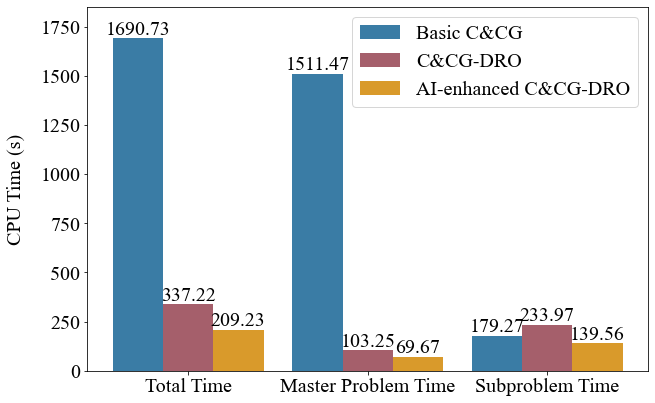} \label{fig:CPU_time}
    \end{minipage}
\end{figure}

Regarding the computational time of individual algorithmic components shown in Figure~\ref{fig:CPU_time}, we note that the primary bottleneck of basic C\&CG lies in solving its master problem. 
This inefficiency arises from its slow convergence, which necessitates many iterations (as reported in the second row of Table \ref{table:algorithm_performance}) and leads to a rapidly expanding master problem, thereby substantially increasing the overall solution time. 
For C\&CG-DRO, although slightly more computational effort is spent on solving the subproblems, the significantly reduced time for the master problems, due to far fewer iterations, renders it remarkably faster than the basic C\&CG.
A notable observation of AI-enhanced C\&CG-DRO is that it incurs the least amount of CPU time in both the master problem and subproblem solves. In particular, for the latter, the time reduction can reach up to 95\% in some instances. This improvement is largely driven by our strong AI tool with the encoding/decoding scheme and WGAN-GP, which produces critical and more representative scenarios (e.g., extreme-point based ones).
As a result, the CG process often yields a smaller scenario set that characterizes the worst-case distribution (as reported in the first row of Table \ref{table:algorithm_performance} below). 
Consequently, the computational burden of both master problems and subproblems is reduced.  

\begin{table}[!ht]
\caption{Summary Statistics of Algorithmic Operations}
\centering
        {\small\def\arraystretch{1.5} 
\begin{tabular*}{0.85\textwidth}{@{\extracolsep{\fill}}lccc}
\hline
\hline
& Basic C\&CG & C\&CG-DRO & AI-enhanced C\&CG-DRO\\
\hline
Avg. Number of Cutting Sets & 261.95 & 178.05 & 167.62\\
Avg. Number of Iterations & 21.70 & 2.83 & 2.74\\
Avg. Final Gap (\%) & 0.42 & 0.00 & 0.00\\
Avg. Time(s)/Iteration & 78.53 & 126.78 & 76.64\\
\hline
\hline
\end{tabular*}
\label{table:algorithm_performance}
}
\end{table}

Finally, in line with the observations from  Figures~\ref{fig:box_CPU_time} and~\ref{fig:CPU_time}, Table~\ref{table:algorithm_performance} shows that the AI-enhanced C\&CG-DRO algorithm consistently outperforms the other two methods across all summary statistics of the major algorithmic operations. 
As a strengthened variant of C\&CG-DRO, it inherits and further enhances the key advantages of the standard C\&CG-DRO approach (e.g., a small number of iterations), while effectively mitigating its key limitations (e.g., high computational cost per iteration). 
In contrast, basic C\&CG not only converges very slowly but also generates a substantially larger number of ineffective scenarios (i.e., cutting sets). 
Based on the aforementioned computational results, the AI-enhanced C\&CG-DRO algorithm is therefore adopted as the solution method in all subsequent experiments.

\subsection{Model Performance and Sensitivity Analysis}
In this subsection, our experimental study primarily focuses on instances of the base case and those obtained from making simple parameter variations around the base case.

\noindent\textbf{Cost and Green Penetration Analysis in the Base Case:} 
From Table~\ref{table:cluster_number} with $S=10$, it can be seen that the capacity adjustment cost plays a dominant role in the strategic cost—205.40 out of 235.65—reflecting that our model is able to adjust production capacities dynamically and economically in response to fluctuating demand.  
In addition, the green-related strategic cost is 19.25, accounting for 12.84\% of the strategic cost. 
As a result, the manufacturing system achieves, even under the worst-case distribution, an average green penetration rate of 11.03\%, satisfying the prescribed green penetration requirement.
Regarding the average service level of 99\%, this outcome is expected. 
Given that the unit penalty cost for unmet demand is lower than the unit production cost in the food and beverage industry, the model derives a least-cost solution that exactly meets the minimum service level.

\noindent\textbf{The Optimal Capacity Plan:}
The optimal capacity plan over the planning horizon is illustrated in Figure~\ref{fig:Decision}, where each pair of digits on the curves represents the numbers of conventional and green production lines, respectively, with their sum equal to the total number of production lines. 
Consistent with our discussion of Table~\ref{table:cluster_number}, all factories will adjust their capacities according to the demands in each time period to minimize costs. 
We also highlight the green technology upgrades and the installation of renewable energy facilities in factories~2 and~3. 
This outcome is expected, as discussed in Section~6.1, because these factories either have lower fixed costs or benefit from more favorable solar conditions.

\begin{figure}[!ht]
        \centering
    \caption{Base Case Planning Results}
 \includegraphics[trim = 100 50 100 0, scale=.34]{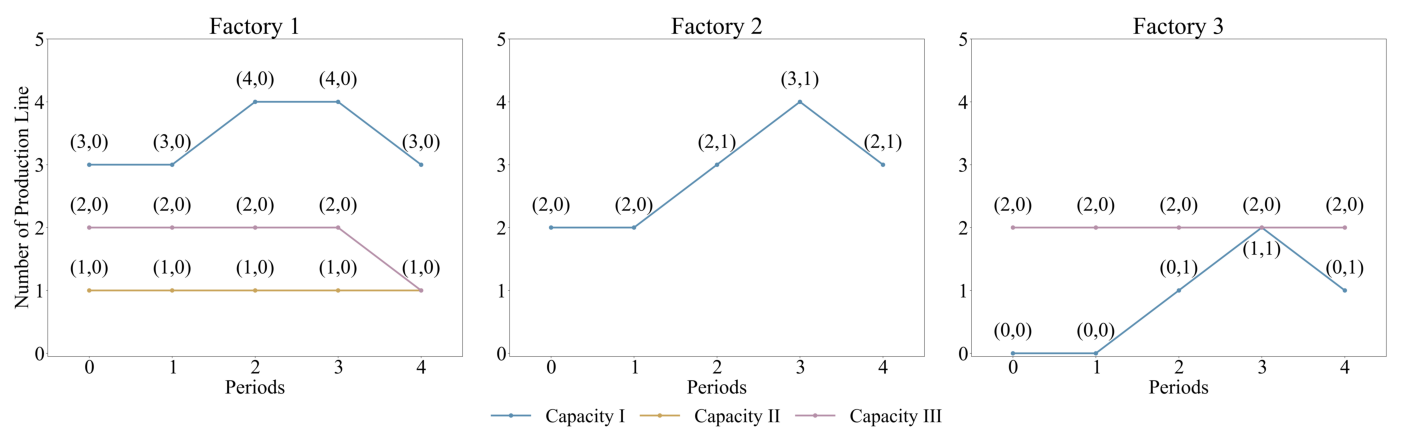}       \label{fig:Decision}
\end{figure}

\noindent\textbf{Sensitivity Analysis on Service Level and Green Penetration Rate:} 
It is noted that the service-level and green-penetration constraints are adjustable, as they are imposed by decision makers to reflect their managerial and sustainability objectives. 
To conduct a systematic sensitivity analysis with respect to these parameters, we first fix the right-hand side of the service-level constraint in \eqref{eq:12} at a target value and vary the right-hand side of the green-penetration constraint in \eqref{eq:10} from 2\% to 18\%. 
We then fix the right-hand side of \eqref{eq:10} at a target value and vary the right-hand side of \eqref{eq:12} from 90\% to 100\%. 
The results are illustrated as box plots in Figure~\ref{fig:Sensitive}. As expected, the total cost increases with both higher service levels and stricter green penetration requirements, consistent with managerial intuition.  Moreover, we observe that the total cost is more sensitive to changes in the service level. 
This can be attributed to the fact that the product capacity adjustment cost constitutes a major portion of the system's strategic cost.

\begin{figure}[!ht]
        \centering
    \caption{Sensitivity Analysis Results}
 \includegraphics[trim = 0 50 0 0, scale=.22]{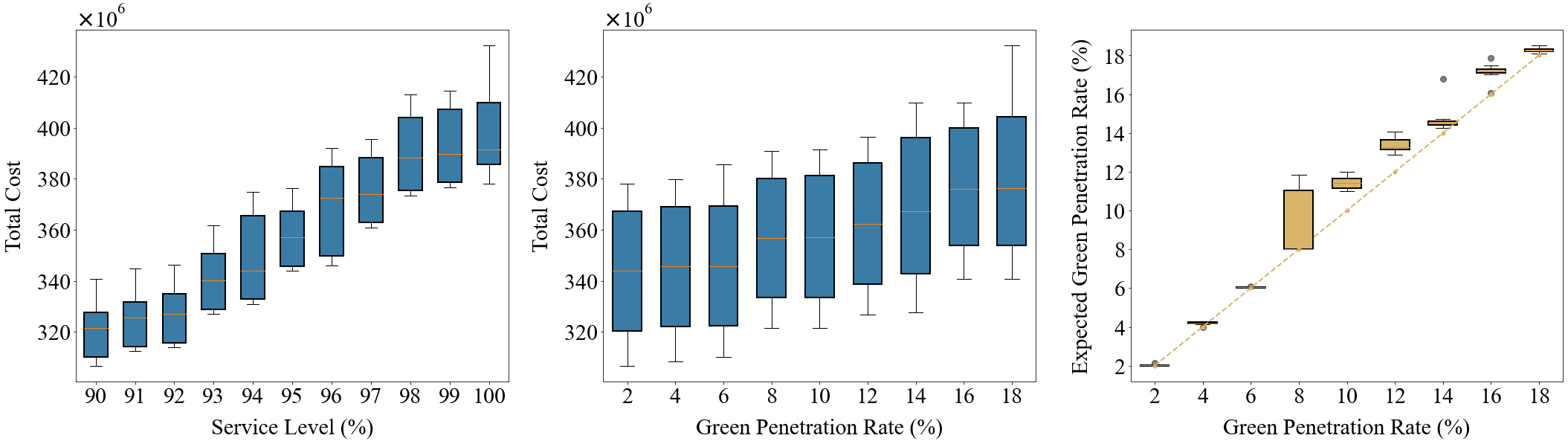}       \label{fig:Sensitive}
\end{figure}

Under all parameter settings, we observe that the manufacturer’s optimal strategy is to meet the service-level requirement exactly.
It is interesting to note that the expected green penetration rate under the worst-case distribution can exceed the imposed requirement significantly, as seen in the last figure in Figure \ref{fig:Sensitive} when that requirement is set between 8\% and 16\%. 
This phenomenon can be attributed to the structural characteristics of the system. 
Note that renewable energy investment decisions are binary at individual factories. 
As a result, satisfying a given penetration requirement may necessitate the actual expected green penetration rate, under all distribution within the ambiguity set, to exceed the predefined requirement.

This observation suggests that policymakers should account for manufacturers’ system structures, demand scale and variability, and the industry's overall risk tolerance when designing green penetration policies.
Failing to incorporate these considerations may result in excessively stringent policy standards that induce over-investment in renewable capacity, leading to inefficiencies and avoidable welfare losses.

\subsection{Performance Comparison Between DRO and the Classical SP Approach}

As noted in Section 2, SP has been widely used as a mainstream method to handle random factors in capacity and production planning.
To demonstrate the applicability and advantages of DRO in this regard, we conduct a systematic comparison between the proposed DRO and the sample average approximation (SAA)-based SP model (presented in \ref{EC4}). 
Specifically, we employ the previously mentioned 100 testing instances for DRO and generate their sampling-based counterparts for SP. 
To enable a consistent and comparable evaluation of their performance, two random-demand distributions are considered when generating scenarios for SP: (i) uniform (Uni.) distributions over the sample space of each cluster, and (ii) truncated Gaussian (Gau.) distributions with the mean set to the center of the sample space and the standard deviation set to one-sixth of the interval size of the sample space for each cluster. 
Hence, both empirical distributions represented by these scenarios lie within the ambiguity set underlying our DRO model, guaranteeing that any solution feasible to DRO is also feasible for SP. 
Nevertheless, the converse does not hold, as solutions feasible to SP may violate feasibility requirements under certain distributions of the ambiguity set. 
Indeed, such violations can occur frequently, posing a substantial challenge for decision-making through SAA-based SP.

\begin{table}[!ht]
\caption{Performance of DRO and SP Models}
\centering
        {\small\def\arraystretch{1.5} 
\begin{tabular*}{0.9\textwidth}{@{\extracolsep{\fill}}l|ccc|cc|cc}
\hline
\hline
& \multicolumn{3}{c}{DRO} & \multicolumn{2}{c}{SP (Uni.)} & \multicolumn{2}{c}{SP (Gau.)}\\
\hline
Demand Distribution & Wor. & Uni. & Gau. & Uni. & Wor. & Gau. & Wor.\\
\hline
Number of Feasible Instances & \textbf{86} & 86/\textbf{86} & 86/\textbf{86} & \textbf{92} & 0/\textbf{92} & \textbf{92} & 37/\textbf{92} \\
Avg. Total Cost & 448.00 & 430.83 & 430.64 & 413.90 & - & 427.43 & 429.81 ($\overline{F}$) \\
Avg. Green Penetration Rate (\%)& 9.91 & 9.86 & 9.83 & 9.47 & - & 9.52 & 7.98 ($\overline{F}$)\\
\hline
\hline
\end{tabular*}
\label{table:model_performance}
}
\end{table}

After solving both the DRO and SP models, the resulting optimal decisions are cross-validated under the worst-case (Wor.) distribution in the ambiguity set as well as the two aforementioned distributions, respectively. 
All computational results are presented in Table \ref{table:model_performance}, where $\overline{F}$ indicates that the associated value represents the average over DRO-feasible instances. Note that, among the 100 testing instances, our DRO model identifies feasible (indeed, optimal) solutions for 86 instances, while the remaining 14 instances are DRO infeasible.
When cross-validating the DRO model’s solutions with respect to uniformly or Gaussian-sampled scenarios, they are 100\% feasible, which is naturally expected.
On the other hand, the SP model demonstrates a poor performance. 
For the 100 instances with uniformly or Gaussian-sampled scenarios, SP derives optimal solutions for two sets of 92 instances, both of which include the previous 86 instances as a subset. Nevertheless, when these solutions are cross-validated with respect to the ambiguity set, they are frequently infeasible. In particular, for the 92 solutions generated by SP with uniformly sampled scenarios, none is feasible, i.e., there exists at least one distribution under which each solution becomes infeasible.

Regarding the average total costs, which are obtained over 86 DRO-feasible instances, we first compare the in-sample performance of the two models under perfect distributional knowledge.
The SP model yields smaller cost estimates under its assumed true distributions, with average values of 413.90 and 427.43 for uniform and Gaussian distributions, respectively, compared to 448.00 for the DRO model under worst-case distributions.
This increase in cost indicates that DRO yields more conservative solutions.
More importantly, we evaluate the out-of-sample (OOS) performance of the DRO solutions by assessing their expected costs across the uniformly and Gaussian-distributed demand scenarios.
Under these OOS scenarios, the DRO decisions achieve average costs of 430.83 and 430.64 respectively. 
In particular, for the Gaussian-sampled scenarios, the performance loss of the DRO solutions relative to the SP benchmark (whose average cost is 427.43 with the exact distributional knowledge) is marginal. 
Moreover, for the 37 instances where SP solutions are feasible under the ambiguity set, the average  cost of DRO solutions (414.18) is clearly lower than that of SP solutions (429.81). 
In addition, the DRO model performs well in meeting the green penetration requirement. By leveraging green technology upgrades and coordinated production across factories over the planning horizon (as shown in Figure \ref{fig:Decision}), its solutions increase PV energy utilization to satisfy both production demand and the green penetration target.
These results show that the DRO solutions simultaneously achieve strong economic performance, sustainability objectives, and desirable out-of-sample robustness.

Before concluding this subsection, we note in our industrial context that the primary sources of infeasibility arise from the green-penetration requirements, the service-level constraints, and, in particular, their interactions. 
As mentioned earlier, these constraints are imposed by decision makers to reflect managerial and sustainability objectives, and the ability of our DRO model and solution algorithm to explicitly detect infeasibility within the ambiguity set is of substantial practical value. 
Such diagnostic capability enables management teams to revise their targets, allocate more resources, or acquire additional information to better understand and reduce underlying uncertainty. So, we believe it should be promoted for broader adoption in real-world decision-making settings where feasibility under uncertainty is as critical as optimality.

\section{Conclusion}
This study focuses on a strategic capacity planning problem of green manufacturing, involving multiple factories, production lines, and product categories. It is complicated by the randomness of the climate factors that affect product demand and renewable energy generation. 
To address this challenge, a rather comprehensive two-stage DRO-based green manufacturing capacity planning model is developed, based on an ambiguity set constructed by a data-driven clustering technique to leverage data with diverse characteristics. To solve practical-scale instances, an effective generative AI network is integrated into the customized C\&CG-DRO computational framework. Particularly, a novel encoding/decoding scheme is designed to provide the AI model with structurally informative data for learning and transform AI-generated outputs into formats that are compatible with C\&CG-DRO. 

Experimental results show that the proposed model is able to achieve both strong economic performance and robust feasibility under random demand and renewable energy generation, and the AI-enhanced framework significantly boosts the computational efficiency and consistency compared to the standard implementation. More broadly, the findings highlight the value of integrating green technology adoption with coordinated capacity planning to better leverage renewable energy while aligning production efficiency with sustainability and corporate social responsibility objectives.

There are two main areas for future research. 
First, the proposed framework can be extended to large-scale and complex manufacturing systems, such as multinational networks managing production and logistics under heterogeneous environmental regulations and carbon policies, thereby supporting robust and low-carbon planning and operations amid regulatory and climate uncertainty. 
Second, the development of more sophisticated encoding/decoding schemes and the integration of additional AI techniques into diverse decision-making frameworks remains a promising direction.
Research in these areas may yield more effective methodological tools for advancing operations and management practice.


\bibliographystyle{ref}

 \let\oldbibliography\thebibliography
 \renewcommand{\thebibliography}[1]{
    \oldbibliography{#1}
    \baselineskip12pt 
    \setlength{\itemsep}{1pt}
 }
\bibliography{ref}

\ECSwitch
\ECHead{E-Companion}

\section{Proofs of Propositions} \label{EC1}

\subsection{Proof of Proposition~\ref{Proposition:A2}}
\proof{Proof.} Notice that $Q^{s}(\boldsymbol{X}^{*},\boldsymbol{\xi}^{s})$ is a linear minimization problem.
Thus \eqref{eq:22} is a bi-level problem. By leveraging the duality of the recourse $Q^{s}(\boldsymbol{X}^{*},\boldsymbol{\xi}^{s})$, one can have
 \begin{equation}\label{eq:ec1}
    \begin{aligned}
        r^{s*} = \max_{\boldsymbol{\xi}^{s},\boldsymbol{\pi}^{s}} & \sum_{\mathcal{K}} \sum_{\mathcal{T}}(\beta^{Ls*}_{kt} - \beta^{Us*}_{kt})\xi_{kt}^{s} - \alpha^{s*} + (\boldsymbol{d} - \boldsymbol{B}^{s}_{\boldsymbol{X}}\boldsymbol{X}^{*} - \boldsymbol{B}_{\boldsymbol{\xi}}\boldsymbol{\xi}^{s})^\top \boldsymbol{\pi}^{s}\\
        \mathrm{s.t.}&\,\xi^{Ls}_{kt} \leq \xi^{s}_{kt} \leq \xi^{Us}_{kt},\forall k \in \mathcal{K},\forall t \in \mathcal{T};\\
        &\boldsymbol{B}_{\boldsymbol{Y}}^\top \boldsymbol{\pi}^{s} \leq \boldsymbol{C}_{\boldsymbol{Y}}; \ \ \boldsymbol{\pi}^{s} \geq 0,
    \end{aligned}
 \end{equation}
where $\boldsymbol{\pi}^{s}$ is a dual variable vector. 

Denote the feasible region of $\boldsymbol{\pi}^{s}$ as $\boldsymbol{\Pi}^{s} = \left \{ \boldsymbol{B}_{\boldsymbol{Y}}^\top \boldsymbol{\pi}^{s} \leq \boldsymbol{C}_{\boldsymbol{Y}}, \boldsymbol{\pi}^{s} \geq 0 \right \}$, which is non-empty and independent of $\boldsymbol{\xi}^{s}$. Note that \eqref{eq:ec1} reduces to a simple LP for any fixed $\pi^s\in \Pi^s$. 

Given the structure of the first constraint of \eqref{eq:ec1}, i.e., the sample space of $\xi^s_{kt}$, it is straightforward to show that \eqref{eq:ec1} has an optimal solution that $\xi^{s*}_{kt}$ equals either $\xi^{Ls}_{kt}$ and $\xi^{Us}_{kt}$.\Halmos\endproof

\subsection{Proof of Proposition~\ref{Proposition:A3}}
\proof{Proof.} 
As shown in the proof of the previous Proposition, the dual feasible region $\boldsymbol{\Pi}^{s}$ is independent of $\boldsymbol{\xi}^{s}$.
To enumerate the extreme points of $\boldsymbol{\Pi}^{s}$ as $\{\boldsymbol{\pi}^{s,n}\}_{n=1}^{N_s}$ (with $N_s<\infty$), the dual objective yields
\begin{equation}
Q^{s}(\boldsymbol{X}^{*},\boldsymbol{\xi}^{s})
=\max_{n=1,\ldots,N_s}
\left\{
\left(\boldsymbol{d}-\boldsymbol{B}^{s}_{\boldsymbol{X}}\boldsymbol{X}^{*}\right)^{\top}\boldsymbol{\pi}^{s,n}
-\left(\boldsymbol{\xi}^{s}\right)^{\top}\boldsymbol{B}_{\boldsymbol{\xi}}^{\top}\boldsymbol{\pi}^{s,n}
\right\}. \nonumber
\end{equation}
Hence, $Q^{s}(\boldsymbol{X}^{*},\boldsymbol{\xi}^{s})$ is the maximum of finitely many affine functions in $\boldsymbol{\xi}^{s}$, which implies that it is piecewise linear in $\boldsymbol{\xi}^{s}$. 

According to relation \eqref{eq:11}, $Y^{U}_{kt}$ can be replaced by other variables. The objective of \eqref{eq:14} can be rewritten as
 \begin{equation}
    \min \sum_{\mathcal{I},\mathcal{J},\mathcal{K},\mathcal{T}}(c^{O}_{ijk}-c^{U}_{k})Y^{OT}_{ijkt} +\sum_{\mathcal{I},\mathcal{J},\mathcal{K},\mathcal{T}}(c^{N}_{ijk}-c^{U}_{k})(Y^{NT}_{ijkt}+Y^{NG}_{ijkt})+\sum_{\mathcal{K},\mathcal{T}}c^{U}_{k}\xi_{kt}, \nonumber
 \end{equation}
now we can eliminate its constraint \eqref{eq:11} and the constraint \eqref{eq:12} can be rewritten as
 \begin{equation}\label{eq:ec2}
    \sum_{\mathcal{I},\mathcal{J}}(Y^{OT}_{ijkt}+Y^{NT}_{ijkt}+Y^{NG}_{ijkt}) \geq \lambda \xi_{kt}.
 \end{equation}

We consider the effect of a change in $\xi^{s}_{kt}$ while keeping all other parameters fixed. 
Under the standing cost assumptions $c^{O}_{ijk} \geq c^{U}_{k}$ and $c^{N}_{ijk} \geq c^{U}_{k}$ for all $i \in \mathcal I$, $j \in \mathcal J$, and $k \in \mathcal K$, all decision variables in the reformulated recourse problem enter the objective with nonnegative coefficients. 
In \eqref{eq:14}, the demand realization $\xi^{s}_{kt}$ enters the model in two ways. First, it appears linearly in the objective function through the term $c^{U}_{k}\xi^{s}_{kt}$. Second, it affects the feasible region via constraint \eqref{eq:ec2}, where an increase in $\xi^{s}_{kt}$ tightens the demand satisfaction requirement. Since all decision variables carry nonnegative costs, satisfying a more stringent constraint can not decrease the optimal objective value. Therefore, for any $\xi^{s}_{kt}$ within the interval $[\xi^{Ls}_{kt}, \xi^{Us}_{kt}]$, an increase in $\xi^{s}_{kt}$ leads to an non-decreasing optimal value of the recourse problem. This establishes that the recourse cost function $Q^{s}(\cdot)$ is monotonically non-decreasing in $\xi^{s}_{kt}$, which is a key property for subsequent algorithm convergence analysis. \Halmos\endproof

\subsection{Proof of Proposition~\ref{Proposition:A4}}
\proof{Proof.} By construction, $\widetilde{Q}^{s}_{f}(\boldsymbol{X}^{*},\boldsymbol{\xi}^{s})$ measures the minimum total constraint violation required to make $Q^{s}(\boldsymbol{X}^{*},\boldsymbol{\xi}^{s})$ feasible under scenario $\boldsymbol{\xi}^{s}$. 
If $\boldsymbol{X}^{*}$ is feasible, then for every $\boldsymbol{\xi}^{s}$ in the support of any distribution $\mathbb{P}^{s}\in\mathcal{P}^{s}$, $Q^{s}(\boldsymbol{X}^{*},\boldsymbol{\xi}^{s})$ admits a feasible solution without relaxation, implying $\widetilde{Q}^{s}_{f}(\boldsymbol{X}^{*},\boldsymbol{\xi}^{s})=0$. Consequently, the worst-case expected value satisfies $\widetilde{w}^{s*}=0$.

Conversely, if $\widetilde{w}^{s*}>0$, then there exists at least one distribution $\mathbb{P}^{s}\in\mathcal{P}^{s}$ under which the expected constraint violation is strictly positive, indicating that $Q^{s}(\boldsymbol{X}^{*},\boldsymbol{\xi}^{s})$ is infeasible for some realizations of $\boldsymbol{\xi}^{s}$. Hence, $\boldsymbol{X}^{*}$ is infeasible.\Halmos\endproof

\subsection{Proof of Proposition~\ref{Proposition:A5}}
\proof{Proof.} 
It follows from \cite{lu_two-stage_2024} that Algorithm 1 terminates with an optimal solution of O-WESP in a finite number of iterations, which is bounded by the number of extreme points of $\Xi^s$. As $\Xi^s$ is a box set, it is clear that O-CGMP terminates in $O(2^{|\mathcal K||\mathcal T|})$ iterations. Next, we prove the tightness of that constraint at termination by contradiction. Without loss of generality, assume that 
$\xi^{Ls}_{kt}< \gamma^{Us}_{kt}< \xi^{Us}_{kt}$, $\forall k \in \mathcal{K}$ and $\forall t \in \mathcal{T}$, otherwise we can remove that constraint or simply set $\boldsymbol{\xi}^{s}_{kt}$ to the upper or lower bound. 
Consider an arbitrary optimal distribution  $P^*_{\boldsymbol{\xi}^s}>0$ for $\boldsymbol{\xi}^s\in \Xi^{s*}$ with
 \begin{equation}
        \sum_{\boldsymbol{\xi}^{s}\in\Xi^{s*}} P^*_{\boldsymbol{\xi}^{s}}\xi_{kt}^{s} < \gamma^{Us}_{kt} \ \forall k \in \mathcal{K},\forall t \in \mathcal{T}. \label{eq_EC_P4}
 \end{equation}
Define the corresponding mean values 
 \begin{equation}
        \mu_{kt}
:=\sum_{\boldsymbol{\xi}^{s}\in\Xi^{s*}} P^{*}_{\boldsymbol{\xi}^{s}}\,\xi^{s}_{kt}, \nonumber
 \end{equation}
and the strictly positive slacks
 \begin{equation}
        \delta_{kt}:=\gamma^{Us}_{kt}-\mu_{kt}>0. \nonumber
 \end{equation}
Denote $\delta_{k't'}=\min_{k,t}\delta_{kt}$, which is strictly larger than 0. 
Clearly, $\gamma_{k't'}^s(\theta) = \sum_{\boldsymbol{\xi}^{s}\in\Xi^{s*}} P^{*}_{\boldsymbol{\xi}^{s}}\,\big(\boldsymbol{\xi}^{s}_{k't'}+(\xi_{k't'}^{Us} - \boldsymbol{\xi}^{s}_{k't'})\theta\big)$ is a continuous function on $\theta \in [0,1]$. Since the range of $\gamma_{k't'}^s(\theta)$ is $[\mu_{k't'}, \xi_{k't'}^{Us}]$ and $\mu_{k't'} < \gamma_{k't'}^{Us} < \xi_{k't'}^{Us}$, there exists a $\theta^*$ such that $\gamma_{k't'}^s(\theta^*) = \gamma_{k't'}^{Us}$. We introduce a new set of scenarios $\tilde{\Xi}^{s^*} = \{\tilde{\boldsymbol{\xi}}_{kt}^s\}$ as in the following:
 \begin{equation}
        \tilde{\boldsymbol{\xi}}^s_{kt}
        =
        \begin{cases}
         \boldsymbol{\xi}^s_{kt}& k\neq k' \mathrm{or} \ t\neq t',\\[0.3em]
         \boldsymbol \xi^{s}_{k't'}+(\xi_{k't'}^{Us} - \boldsymbol{\xi}^{s}_{k't'})\theta^* & k=k' \mathrm{and} \ t=t'.
        \end{cases}  \label{ECP_constuction}
 \end{equation}
Since $(\xi_{k't'}^{Us} - \boldsymbol{\xi}^{s}_{k't'})\theta^* \geq 0$, according to Proposition \ref{Proposition:A3} and \eqref{ECP_constuction}, we have the following relationship immediately
 \begin{equation}
         \sum_{\boldsymbol{\xi}^{s}\in \tilde{\Xi}^{s*}}P^{*}_{\boldsymbol{\xi}^{s}}
        Q^{s}(\boldsymbol{X}^{*},\boldsymbol{\xi}^{s})\geq \sum_{\boldsymbol{\xi}^{s}\in\Xi^{s*}} P^{*}_{\boldsymbol{\xi}^{s}}\,Q^s(\boldsymbol{X}^{*},\xi^{s}). \nonumber
 \end{equation}
 Hence, for any arbitrary optimal distribution, there exists another optimal distribution $P^*_{\boldsymbol{\xi}^s}$ supported on $\boldsymbol{\xi}^s \in \tilde{\Xi}^{s^*}$ such that $\sum_{\boldsymbol{\xi}^{s}\in \tilde{\Xi}^{s*}} P_{\boldsymbol{\xi}^{s}}\xi_{k't'}^{s} =\gamma^{Us}_{k't'}$. Therefore, there must exist at least one $(k,t)$ such that
\begin{equation}
        \sum_{\boldsymbol{\xi}^{s}\in\Xi^{s*}} P_{\boldsymbol{\xi}^{s}}\xi_{kt}^{s} =\gamma^{Us}_{kt}. \nonumber
\end{equation}
\Halmos\endproof

\section{Model Reformulations} \label{EC2}

\subsection{Reformulation of PP}
Since the inner recourse problem in PP is a LP, we can adopt the KKT conditions to reformulate \eqref{eq:ec1} as
 \begin{equation}
    \begin{aligned}
        \max_{\boldsymbol{\xi}^{s},\boldsymbol{\pi}^{s}} & \sum_{\mathcal{K}} \sum_{\mathcal{T}}(\beta^{Ls*}_{kt} - \beta^{Us*}_{kt})\xi_{kt}^{s} - \alpha^{s*} + \boldsymbol{C}^\top_{\boldsymbol{Y}}\boldsymbol{Y}^{s}\\
        \mathrm{s.t.}&\,\xi^{Ls}_{kt} \leq \xi^{s}_{kt} \leq \xi^{Us}_{kt};\forall k \in \mathcal{K},\forall t \in \mathcal{T},\\
        &\boldsymbol{B}_{\boldsymbol{Y}}^\top \boldsymbol{\pi}^{s} \leq \boldsymbol{C}_{\boldsymbol{Y}},\\
        &\boldsymbol{Y}^{s} \perp (\boldsymbol{C}_{\boldsymbol{Y}}-\boldsymbol{B}_{\boldsymbol{Y}}^\top \boldsymbol{\pi}^{s}) = 0,\\
        &\boldsymbol{\pi}^{s} \perp (\boldsymbol{B}_{\boldsymbol{Y}}\boldsymbol{Y}^{s}+\boldsymbol{B}_{\boldsymbol{\xi}}\boldsymbol{\xi}^{s}-\boldsymbol{d}+\boldsymbol{B}^{s}_{\boldsymbol{X}}\boldsymbol{X}^{*}) = 0,\\
        &\boldsymbol{Y}^{s} \geq 0, \boldsymbol{\pi}^{s} \geq 0.
    \end{aligned} \nonumber
 \end{equation}
Note that the complementary constraints are expressed in the form of products equal to zero, which can be easily linearized by introducing big-$M$ constants and binary decision variables.

\subsection{Model of Feasibility Check Problem}
With the same mechanism of conducting the O-CGMP and PP, the F-CGMP and F-PP share the similar structures. 
The F-CGMP is detailed as
 \begin{equation}
    \begin{aligned}
        \text{F-CGMP}:\widetilde{\eta}^{s}(\widetilde{\Xi}^{s*}) = \max &\sum_{\boldsymbol{\xi}^{s}\in\widetilde{\Xi}^{s*}} P_{\boldsymbol{\xi}^{s}}\widetilde{Q}_{f}^{s}(\boldsymbol{X}^{*},\boldsymbol{\xi}^{s})\\
        \mathrm{s.t.}&\sum_{\boldsymbol{\xi}^{s}\in\widetilde{\Xi}^{s*}} P_{\boldsymbol{\xi}^{s}} = 1,\\
        &\sum_{\boldsymbol{\xi}^{s}\in\widetilde{\Xi}^{s*}} P_{\boldsymbol{\xi}^{s}}\xi_{kt}^{s} \leq \gamma^{Us}_{kt}; \forall k \in \mathcal{K},\forall t \in \mathcal{T},\\
        &-\sum_{\boldsymbol{\xi}^{s}\in\widetilde{\Xi}^{s*}} P_{\boldsymbol{\xi}^{s}}\xi_{kt}^{s} \leq -\gamma^{Ls}_{kt}; \forall k \in \mathcal{K},\forall t \in \mathcal{T},\\
        &P_{\boldsymbol{\xi}^{s}} \geq 0;\forall \boldsymbol{\xi}^{s}\in\widetilde{\Xi}^{s*},
    \end{aligned} \nonumber
 \end{equation}
note that $\widetilde{\Xi}^{s*}$ is the collection of previously identified scenarios. Let $\boldsymbol{I}$ be an identity matrix with appropriate dimension, the F-PP is
 \begin{equation}
    \begin{aligned}
        \text{F-PP}:\widetilde{r}^{s*} = \max_{\boldsymbol{\xi}^{s},\boldsymbol{\pi}^{s}} & \sum_{\mathcal{K}} \sum_{\mathcal{T}}(\beta^{Ls*}_{kt} - \beta^{Us*}_{kt})\xi_{kt}^{s} - \alpha^{s*} + (\boldsymbol{d} - \boldsymbol{B}^{s}_{\boldsymbol{X}}\boldsymbol{X}^{*} - \boldsymbol{B}_{\boldsymbol{\xi}}\boldsymbol{\xi}^{s})^\top \boldsymbol{\pi}^{s}\\
        \mathrm{s.t.}&\,\xi^{Ls}_{kt} \leq \xi^{s}_{kt} \leq \xi^{Us}_{kt};\forall k \in \mathcal{K},\forall t \in \mathcal{T},\\
        &\boldsymbol{B}_{\boldsymbol{Y}}^\top \boldsymbol{\pi}^{s} \leq \boldsymbol{0},\\
        &\boldsymbol{I} \boldsymbol{\pi}^{s} \leq \boldsymbol{1},\\
        &\boldsymbol{\pi}^{s} \geq 0. \nonumber
    \end{aligned}
 \end{equation}
Now the procedures of Algorithm 1 can be used to solve the F-WESP. 

\subsection{Reformulation of MP}
By exploiting the dual formulations of the two embedded maximization problems in MP, it can be equivalently reformulated into a mixed-integer linear program by replacing constraints \eqref{eq:24c} and \eqref{eq:24d}, which is
 \begin{subequations}
    \begin{align}
        \text{MP}:\underline{g} = \min &\, \boldsymbol{C}_{\boldsymbol{X}}^\top \boldsymbol{X} + \sum_{\mathcal{S}} q_{s}\eta_{s} \nonumber\\
        \mathrm{s.t.}&\, \boldsymbol{X} \in \mathcal{X};\nonumber\\
        &\eta_{s} \geq \alpha^{s(o)} + \boldsymbol{\gamma}^{Us\top}\boldsymbol{\beta}^{Us(o)} - \boldsymbol{\gamma}^{Ls\top}\boldsymbol{\beta}^{Ls(o)},\forall s \in \mathcal{S};\nonumber\\
        &\alpha^{s(o)} + \boldsymbol{\xi}^{s\top}\boldsymbol{\beta}^{Us(o)} - \boldsymbol{\xi}^{s\top}\boldsymbol{\beta}^{Ls(o)} \geq \boldsymbol{C}_{\boldsymbol{Y}}^\top \boldsymbol{Y}^{s}_{\boldsymbol{\xi}^{s}},\forall s \in \mathcal{S},\forall \boldsymbol{\xi}^{s} \in \Xi^{s(o)};\nonumber\\
        &\boldsymbol{Y}^{s}_{\boldsymbol{\xi}^{s}} \in \mathcal{Y}^{s}(\boldsymbol{X},\boldsymbol{\xi}^{s}),\forall s \in \mathcal{S},\forall \boldsymbol{\xi}^{s} \in \Xi^{s(o)};\nonumber\\
        &0 \geq \alpha^{s(f)} + \boldsymbol{\gamma}^{Us\top}\boldsymbol{\beta}^{Us(f)} - \boldsymbol{\gamma}^{Ls\top}\boldsymbol{\beta}^{Ls(f)},\forall s \in \mathcal{S};\nonumber\\
        &\alpha^{s(f)} + \boldsymbol{\xi}^{s\top}\boldsymbol{\beta}^{Us(f)} - \boldsymbol{\xi}^{s\top}\boldsymbol{\beta}^{Ls(f)} \geq  \boldsymbol{1}^{\top} \widetilde{\boldsymbol{Y}}^{s}_{\boldsymbol{\xi}^{s}},\forall s \in \mathcal{S},\forall \boldsymbol{\xi}^{s} \in \Xi^{s(f)}; \nonumber\\
        &\boldsymbol{Y}^{s}_{\boldsymbol{\xi}^{s}},\widetilde{\boldsymbol{Y}}^{s}_{\boldsymbol{\xi}^{s}} \in \widetilde{\mathcal{Y}}^{s}(\boldsymbol{X},\boldsymbol{\xi}^{s}),\forall s \in \mathcal{S},\forall \boldsymbol{\xi}^{s} \in \Xi^{s(f)}; \nonumber\\
        &\boldsymbol{\beta}^{Us(o)},\boldsymbol{\beta}^{Ls(o)},\boldsymbol{\beta}^{Us(f)},\boldsymbol{\beta}^{Ls(f)} \geq 0, \nonumber
    \end{align}
 \end{subequations}
where $\alpha^{s(o)},\alpha^{s(f)},\boldsymbol{\beta}^{Us(o)},\boldsymbol{\beta}^{Ls(o)},\boldsymbol{\beta}^{Us(f)}$ and $\boldsymbol{\beta}^{Ls(f)}$ denote the dual variables.

\setcounter{figure}{0}
\renewcommand{\thefigure}{EC.\arabic{figure}}

\setcounter{table}{0}
\renewcommand{\thetable}{EC.\arabic{table}}

\section{Supplementary on Data and Training} \label{EC3}
\subsection{Supplementary on Data}
The historical seasonal peak sunshine hours data are presented in Figure~\ref{fig:sunshine}.
Region~2 experiences the weakest seasonal fluctuations, while region~3 exhibits the most favorable solar irradiance conditions. Overall, as noted in Section 3.3, the entire climate data set is  partitioned into $S$ clusters using the $k$-means clustering method.
\begin{figure}[!ht]
        \centering
    \caption{Historical Sunshine Hours}
 \includegraphics[trim = 0 0 0 0, scale=.30]{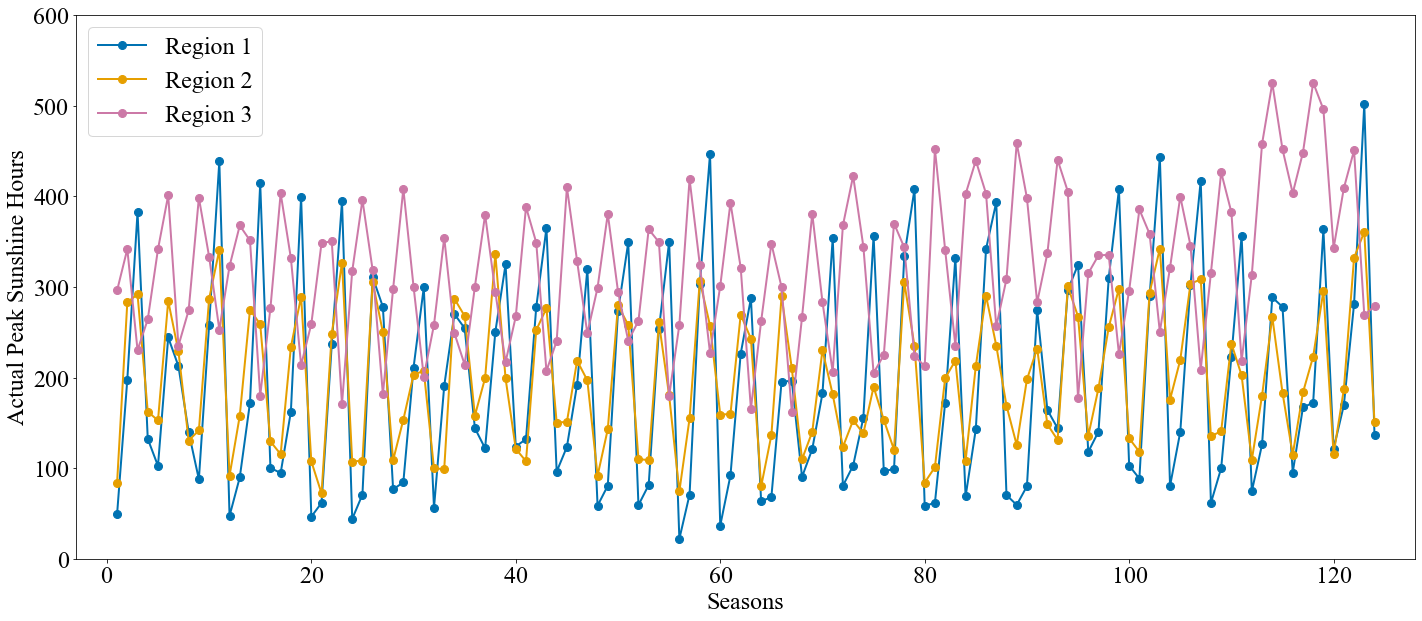}        \label{fig:sunshine}
\end{figure}

Next, we describe how the parameters of the ambiguity set are determined for each cluster.
To facilitate this study, the manufacturer provided a set of nominal demand data, presented in Table \ref{table:demand_baseline}. These data reflect the decision maker’s belief that when solar irradiance is at its historical average level, the nominal demand corresponds to the expected product demand for a given region and time period. Due to confidentiality constraints, additional proprietary demand data were not made available.

\begin{table}[!ht]
\caption{Nominal Demand ($\times10^{4}$)}
\centering
        {\small\def\arraystretch{1.5} 
\begin{tabular*}{0.6\textwidth}{@{\extracolsep{\fill}}lccccc}
\hline
\hline
Region & Product & Quarter 1 & Quarter 2 & Quarter 3 & Quarter 4\\
\hline
\multirow{3}{*}{Region 1} & A & 319.52 & 538.68 & 745.55 & 546.87\\
~ & B & 301.60 & 508.46 & 703.73 & 516.20\\
~ & C & 287.49 & 484.68 & 670.81 & 492.05\\
\multirow{3}{*}{Region 2} & A & 131.83 & 222.25 & 307.60 & 225.63\\
~ & B & 79.68 & 134.34 & 185.93 & 136.38\\
~ & C & 75.28 & 126.91 & 175.65 & 128.84\\
\multirow{3}{*}{Region 3} & A & 110.06 & 185.56 & 256.81 & 188.38\\
~ & B & 56.67 & 95.55 & 132.24 & 97.00\\
~ & C & 83.35 & 140.52 & 194.48 & 142.66\\
\hline
\hline
\end{tabular*}
\label{table:demand_baseline}
}
\end{table}

To perform numerical demonstrations, we use the values in Table \ref{table:demand_baseline} as a baseline to generate demand data across regions, products, and quarters. Specifically, we randomly sample from a predefined parameter space using a mixture of uniform and Gaussian distributions for each region-product-quarter combination. 
Note that the parameters of these distributions are centered around the corresponding nominal demand values and allowed for moderate variability (e.g., making different levels of perturbations on demand data for different quarters), which is reflecting empirical observations and maintaining consistency with realistic demand magnitudes.
Then, those data are assigned to clusters based on the associated  region and quarter. Consequently, the sample space interval for product $k$ in each cluster, $[\xi^{Ls}_{kt},\xi^{Us}_{kt}]$, is determined directly from the empirical minimum and maximum values of the demand data within cluster $s$. In addition, the interval bounds for the first moment of demand, $[ \gamma^{Ls}_{kt}, \gamma^{Us}_{kt}]$, are constructed using the empirical 10th and 90th percentiles of the corresponding demand data. 

We note that the sample space and the interval bounds for the first moment of demand can be readily modified to reflect different data accessibility or the decision maker’s risk attitude, which does not affect the validity of the methodologies presented in the paper, including the encoding/decoding scheme and WGAN-GP.

\subsection{Network Training Configurations} 
Data encoding/decoding and network model training are conducted on a server with an Intel Xeon 4210R processor, 64 GB of RAM, and an NVIDIA RTX A4000 GPU.
In addition to the aforementioned base case, a total of 5,100 instances are generated for training and testing purposes through random perturbations of cost coefficients, green penetration rates, and parameters of the ambiguity set of the base case. Specifically, to generate a new instance, each cost coefficient is obtained by multiplying the base case value by a random factor drawn from $[0.8,1.2]$, the green penetration rate is set to a random number between 1\% and 20\%, and the ambiguity set parameters are scaled by a factor between 1 and 4 relative to their base case counterparts. 

For the network training described in Section~5, 5,000 such single-cluster instances are generated, where their ambiguity sets are instantiated by using a randomly chosen historical sunshine record as the corresponding $h_{it}$ values. 
These small-scale instances can be solved efficiently to construct the original dataset for data encoding, which yields $61,055$ binary images with size $50\times12$. 
The network is then trained adversarially using paired observations of feature vectors and corresponding images. 
The training runs for $300$ epochs with a batch size of $512$.
The noise and feature vectors both have a length of $50$.
We use the Adam optimizer with its momentum parameters $\beta = (0.0, 0.999)$ for the generator and $\beta = (0.5, 0.999)$ for the critic.
The learning rate is set to $10^{-5}$ for both networks.
For each training iteration, the critic is updated four times per generator update to enhance its ability to distinguish real from generated data. 
Its loss function incorporates a gradient penalty term to enforce Lipschitz continuity, enhancing training stability. 
The gradient penalty coefficient is set to $10$.
The generator is then updated to maximize the critic’s Wasserstein evaluation of its outputs, learning to produce samples that are statistically indistinguishable from the underlying data distribution.

The remaining 100 instances are used for testing and are generated by using the complete historical sunshine records and by setting the number of clusters to 10. 
As noted in Sections~5 and~6, these testing instances serve as our primary platform to evaluate the two-stage DRO model, compare the AI-enhanced C\&CG-DRO algorithm with existing methods, and conduct subsequent performance analyses.

\section{Two-stage SAA-based SP Model} \label{EC4}
The general SP model can be written as follows, where $\mathbb{P}^{s}$ denote the distribution of $\boldsymbol{\xi}^{s}$ in cluster $s\in \mathcal S$:
 \begin{equation}
     \min_{\boldsymbol{X} \in \mathcal{X}} \boldsymbol{C}^\top_{\boldsymbol{X}}\boldsymbol{X} + \sum_{\mathcal{S}}q^{s}\mathbb{E}_{\mathbb{P}^{s}} \left[ Q^{s}(\boldsymbol{X},\boldsymbol{\xi}^{s}) \right].\nonumber
 \end{equation}
 
With sample average approximation (SAA) method, we can simplify the complex integration operations. 
Specifically, for cluster $s$, let $\mathcal{M}^s$ denote the set of sampled  scenarios in $s$ and $\boldsymbol\xi^s_m$ an element of $\mathcal M^s$. The SAA-based two-stage SP model is: 
 \begin{equation}
     \min_{\boldsymbol{X} \in \mathcal{X},\boldsymbol{Y}^{s}_{m} \in \mathcal{Y}^{s}(\boldsymbol{X},\boldsymbol{\xi}^{s}_{m})} \boldsymbol{C}^\top_{\boldsymbol{X}}\boldsymbol{X} + \sum_{\mathcal{S}}\frac{q^s}{|\mathcal M^s|}\sum_{\mathcal{M}^s} \boldsymbol{C}_{\boldsymbol{Y}}^\top \boldsymbol{Y}^{s}_{m}. \nonumber
 \end{equation}
Note that the detailed form of $\mathcal{Y}^{s}(\boldsymbol{X},\boldsymbol{\xi}^{s}_{m})$ is consistent with the formulation given in equation \eqref{eq:19} by introducing decision variable $\boldsymbol{Y}^{s}_m$ for every scenario $\boldsymbol{\xi}^{s}_m$.
In our computational experiments, $|\mathcal M^s|$ is set to 100. 
These samples are randomly selected from the sample spaces of each of the 10 clusters, according to a predefined distribution.

\end{document}